\documentclass[12pt]{article}
\usepackage[a4paper, total={6in, 8in}]{geometry}
\usepackage{amsmath}
\usepackage{amsfonts}
\usepackage{amsthm}
\usepackage{hyperref}
\usepackage{bbm}
\usepackage{scrextend}
\usepackage{amssymb}
\usepackage{graphicx}
\graphicspath{ {./images/} }
\newcommand*{\p}{\mathbb{P}}
\usepackage{xcolor}
\usepackage{dsfont}
\usepackage{natbib}
\usepackage{graphicx}
\graphicspath{ {./images/} }
\usepackage{float}
\usepackage{multibib}
 \usepackage{algorithm} 
\newcites{SM}{References}

\newcommand{\floor}[1]{\lfloor #1 \rfloor}
\newcommand{\ceil}[1]{\lceil #1 \rceil}

\newcommand{\norm}[1]{\left\lVert#1\right\rVert}
\usepackage{setspace}

\newtheorem{Lemma}{Lemma}[section]

\newtheorem{theorem}[Lemma]{Theorem}

\newtheorem{remark}[Lemma]{Remark}
\newtheorem{example}[Lemma]{Example}
\usepackage[utf8]{inputenc}

\date{}

\definecolor{darkblue}{rgb}{.1, 0.1,.8}
\definecolor{darkgreen}{rgb}{0,0.8,0.2}
\definecolor{darkred}{rgb}{.8, .1,.1}

\usepackage{mathtools}
\mathtoolsset{showonlyrefs}
\newcommand*{\E}{\mathbb{E}}
\newcommand*{\N}{\mathbb{N}}
\newcommand*{\R}{\mathbb{R}}

\newcommand{\1}{\mathbbm{1}}

        \title{Gradual changes in functional time series}
\author{
  {\normalsize Patrick Bastian, Holger Dette} \\
{\normalsize  Ruhr-Universit\"at Bochum} \\
{\normalsize  Fakult\"at f\"ur Mathematik} \\
{\normalsize  44780 Bochum, Germany}
}

\date{}

\begin{document}
     
        \maketitle

        \begin{abstract}
We consider the problem of detecting gradual changes in the sequence of mean functions from  a not necessarily stationary functional time series. Our approach is  based on   the maximum deviation (calculated over a given time interval) between a benchmark function  and the   mean functions at  different time points.  We speak of a gradual change of size $\Delta  $, if this quantity exceeds  a given   threshold $\Delta>0$. For example, the benchmark function could represent an average of yearly  temperature curves from  the pre-industrial time,  and we are interested in the question if the  yearly temperature curves afterwards deviate from the pre-industrial average by more than $\Delta =1.5$  degrees Celsius, where the deviations are  measured with respect to the sup-norm. Using Gaussian approximations for high-dimensional data we develop a test for hypotheses of this type and estimators for the time where a deviation of size larger than $\Delta$ appears for the first time.  We  prove the validity of our approach and  illustrate the new methods  by a simulation study and a data example, where we analyze yearly temperature curves at different stations in Australia. 
        \end{abstract}



\section{Introduction}
  \def\theequation{1.\arabic{equation}}	
  \setcounter{equation}{0}
In recent decades change point detection has found considerable interest in the literature. Its  applications are numerous and include a diverse menagerie of subjects such as economics, hydrology, climatology, engineering, genomics and linguistics to name just a few. Surveys summarizing some of the more recent results include among others \cite{Aue2013}, \cite{TRUONG2020} and \cite{Aue2023}. We also refer the interested reader to \cite{Horvath24} for a comprehensive overview. In this context data is often given as a uni- or multivariate time series and one then tries to choose an appropriate model to prove mathematically  the validity of inference tools for  detecting  structural breaks in the time series. In the simplest case one considers a univariate  time series $(X_j)_{1 \leq j \leq n}$ of the form 
\begin{align}
\label{p50}
    X_j=\mu(j/n)+\epsilon_j~.
\end{align}
Here $(\epsilon_j)_{1 \leq j \leq n}$ is a stationary error process while $\mu:[0,1]\rightarrow \R$ is a piecewise constant function and the statistical task then consists of estimating $\mu$. A large part of the literature is concerned with the case of only one change (often abbreviated AMOC - at most one change), but the detection of multiple change points has also received considerable interest (see, for example   \cite{Fryzlewicz2013}, \cite{Frick2014},  
\cite{baranowskietal2019}, \cite{Eckle2020} and the references in these works).

While the above setting with a locally constant mean is often well justified in applications and also theoretically attractive,  there are certain classes of applications such as temperature or financial data where justifying this assumption is difficult. In these cases a model with a smoothly varying mean $\mu$ might be more appropriate even if the theoretical analysis of the associated statistical procedures become substantially more involved. In such cases one talks about {\it  gradually changes}   and 
 some progress in this direction has been made by \cite{Huskova2000},  \cite{Vogt2015}, \cite{Dette2019b} and \cite{buecher21}).

Nowadays, in  the big data era more complex dependent data structures such as functional time series are recorded (see \cite{Horvath2012} for a comprehensive overview)  and just as in the Euclidean case change point detection is a central concern in this setting (see e.g. \cite{Berkes2009}, \cite{Aue2017}, \cite{dette2020}, \cite{HORVATH2022}). Meanwhile, there is a substantial amount of literature for the AMOC case in the functional setting and  several authors have also worked  on detecting  multiple change points  in functional time series (e.g. \cite{Padilla2022}, \cite{rice2022} and \cite{bastian2023}). On the other hand, to the best of our knowledge,  change point analysis for gradually changing functional time series is not available. 

In  the present paper we address this challenging problem and develop statistical inference tools  for detecting gradual changes in the mean function of a non-stationary   functional times series of the form
\begin{align} \label{det1}
    X_j(s)=\mu(j/n,s)+\epsilon_j(s) \quad \quad j=1,...,n
\end{align}
where $(\epsilon_j(\cdot))_{1 \leq j \leq n}$ is a mean zero  not necessarily stationary process and $\mu(\cdot,\cdot)$ is a sufficiently smooth mean function. Our modeling of  functional data in the form \eqref{det1} is motivated by the observation that in many  cases it is very unlikely that the mean functions  $\mu (t,\cdot )$ will be exactly constant over certain time intervals.  In fact  we argue that there will  always be a difference between the functions $s \to \mu (t_1,s) $ and 
$s \to \mu (t_2,s) $ at different  time points $t_1$ and $t_2$
in some  decimal place. Therefore, it is more natural to assume that these functions are smoothly changing. 
Similarly, it is hard to justify that the error  process $(\epsilon_j(\cdot))_{1 \leq j \leq n}$ is  stationary.  

Model  \eqref{det1} reflects this point of view. However, as the mean function is continuously changing, a change point problem cannot be formulated in  its ``classical'' form $H_0:\mu_1=\mu_2 \text{ vs } H_1:\mu_1 \neq \mu_2$, where $\mu_1$ and $\mu_2$ are the functions before and after a change point, respectively. 
Instead we consider the problem of detecting a  deviation of the mean from a given benchmark function such as the initial mean function $\mu(0,\cdot)$ at "time $0$"  or the average mean function calculated from the data in a specified time frame in the past.  More precisely we are interested in hypotheses of the form
\begin{align}
\label{p51}
    H_0(\Delta):d_\infty :=\sup_{(t,s)}|\mu(t,s)-g_\mu(s)|\leq \Delta \quad \text{ vs } \quad H_1(\Delta):d_\infty> \Delta ~. 
\end{align}
where $\Delta > 0$ is a pre-specified threshold. We will develop a test for hypotheses of this type and  construct an estimator for the first time of such a deviation, that is 
$$
t^* = \inf \big \{ t |~ \sup\limits{_s}   | \mu(t,s)-g_\mu(s)  |\leq \Delta  \big \} 
.
$$
A noteworthy advantage of hypotheses of this form is their easy interpretability for practitioners as soon as a threshold function $g_\mu$ and a threshold $\Delta >0$ have been chosen. This choice is application specific.    Consider, for example  the case of analyzing annual temperature data where each observation is a curve representing the temperature over a year (or another time span) at a specific location. One may choose the function $g_\mu$ as  the historical average of   pre-industrial temperature curves, say $g (s) =\int_{0}^{x_0} \mu (u,s)  du$ and then calculate the supremum in \eqref{p51} over the set $\{ (t,s) |  t \geq x_0 \} $. A typical threshold  in such an application is   $\Delta = 1.5$ degrees Celsius corresponding to the  Paris Agreement  adopted  at the UN Climate Change Conference (COP21) in Paris, 2015.
In this case testing the hypotheses \eqref{p51} automatically yields a coherent answer to questions of the type  {\it "has the temperature increased by more than $\Delta$ degrees  Celsius compared the pre-industrial temperatures"}. We also develop a data adaptive  rule for the choice of $\Delta $ to  define a  measure of evidence  for a  deviation of $\mu$ from  the benchmark function $g_\mu$   with a controlled type I error  $\alpha$,  see Remark \ref{pr2} for details.

These advantages come with the cost of substantial mathematical challenges  in the construction of corresponding change point analysis tools  and 
their theoretical analysis to obtain statistical guarantees. The problems are already apparent in the recent  work by  \cite{buecher21} who consider hypotheses of the form \eqref{p51} for univariate data (corresponding to a fixed value of  $s$ in the our  case). For the univariate model \eqref{det1}  with a stationary error process these authors proposed a nonparametric  estimator 
for the maximum deviation $\sup_{t \in [0,1]} | m(t) -g| $ (here $g \in \mathbb{R}$ is a benchmark) and investigated its  asymptotic distribution. Among other things their asymptotic analysis  relies crucially on two results which  are not available in the functional setting considered in  this paper:
\begin{itemize}
    \item [(i)] A strong approximation for partial sums of weakly dependent random variables with quickly decaying approximation error.
    \item [(ii)] Results on the limits of maxima of stationary Gaussian fields (see, for example,  \cite{Piterbarg1996}).
\end{itemize}
Regarding (i)  we remark that, the analysis of nonparametric estimators for the distance $d_\infty $ defined in \eqref{p51} would require a strong approximation result for partial sums of functional time series. 
We are  able to obtain such results  for $\beta-$mixing $C[0,1]$-valued random variables using results from  \cite{Dehling1983}. However, the error in these approximations does not decay sufficiently fast  for our purposes.  The situation is even worse in regards to (ii), where the functional nature of the error process introduces heavy non-stationarities in the stochastic approximations of the estimator. This makes the derivation of  limit results for a regression approach virtually impossible in our setting (note that  determining the limit of the maximum of a non stationary Gaussian sequence is already extremely challenging in the univariate case).

To solve these issues we will  pursue  an alternative approach and 
develop a novel bootstrap procedure which is applicable under mild assumptions on the serial dependence and on the smoothness of the error process and of the function $(t,s) \to \mu(t,s)$. In particular, in contrast to \cite{buecher21}, the methodology developed in this paper does not require the assumption of a second order stationary error process in model \eqref{det1}. In 
Section \ref{sec2} we will formally introduce the testing problem \eqref{p51} in more detail and propose an estimator for the maximum deviation $d_\infty$. We show that this estimator enjoys a stochastic expansion which heavily depends on the extremal points of the maximum deviation (i.e. the points $(t,s)$ where the maximum absolute deviation between the mean function $\mu(t,s)$ and the benchmark function $g_\mu(s)$ is attained). We use this result to develop a bootstrap procedure that is able to consistently test the introduced hypotheses at nominal  level $\alpha$. In Section \ref{sec3} we construct an estimator for the time of the first relevant deviation. To evaluate the finite sample performance we present a simulation study in Section \ref{sec4} which also contains an application of our methodology to real world climate data. Finally, in  Section \ref{sec5} we give the proofs of our main  results.   

\section{Gradual changes in functional data}
  \def\theequation{2.\arabic{equation}}	
  \setcounter{equation}{0}
\label{sec2}

Let  $C (T)$ denote the space of continuous functions defined on the set $T$. Throughout this paper $T$ will be of the form $T=\prod_{i=1}^k [a_i,b_i]$ for some $0 \leq a_i < b_i \leq 1$. We define by $\|  \cdot \|_{\infty,T} $ the sup-norm on $C(T)$ and use the notation  $\|\cdot\|_\infty$ whenever  the corresponding space $C(T)$ is clear from the context.
We consider a triangular-array $\{ X_{n, j} |  \, j = 1, \dots, n \}_ {n \in \N}$ of $C ( [0,1])$-valued random variables with the representation
\begin{align}\label{p1}
    X_{n, j}(s) = \mu\Big(\frac{j}{n},s\Big) + \epsilon_{j}(s),  \quad \quad j = 1, \dots, n,  
\end{align}
where $\{\epsilon_{j} \}_{j \in \N} $  is a mean zero process in $C([0,1])$ and $\mu \in C ( [0,1]^2)$ denotes the unknown  time dependent mean function. Note that we assume at this point that the full trajectories $s \rightarrow X_{n,j}(s)$ are available. The results can be extended to densely observed functional time series  (see Remark \ref{pr1} below). Note  that we do not assume that the error process $\{\epsilon_{j} \}_{j \in \N} $ is second order stationary.\\

We are interested in detecting significant
deviations in the sequence of mean functions $\{ \mu(j/n,\cdot)\}_{j=1,...,n}\subset C([0,1])$ from a given benchmark function in a time interval $[x_0,x_1] \subset [0, 1]$ (see Example \ref{pex1} below for some concrete examples), where we measure deviations with respect to the supremum norm. For this purpose, we consider a continuous function $g_\mu:[0,1]\rightarrow \R$ as a benchmark and define the distance
\begin{align}
\label{p2}
    d_\infty:=\sup_{t \in [x_0,x_1]} \sup_{s \in [0,1]}| \mu(t,s)-g_\mu(s)|~.
\end{align}
For the sake of simplicity we will not reflect the dependence of $d_\infty$ on $x_0,x_1$ and $g_\mu$ in our notation. We are interested in the hypotheses 
\begin{align}
\label{p3}
    H_0(\Delta): d_\infty \leq \Delta \quad \text{ vs } \quad H_1(\Delta):d_\infty > \Delta ,
\end{align}
where $\Delta > 0$ is a given constant, and in the time of the first time $t$ where the maximum deviation of the function  $s \to \mu (t,s)$ from the benchmark function $s \to  g_\mu (s)$ is larger or equal than $\Delta$, that is
\begin{align}
\label{p200}
    t^*(\Delta):=\inf\{ t \ | \ \|\mu(t,\cdot)-g_\mu(\cdot)\|_\infty \geq \Delta \}~.
\end{align}
Before we continue, we discuss two  choices for the benchmark function  $g_\mu$.
\begin{example}
\label{pex1} ~~ {\rm 
   \begin{enumerate}
    \item[(1)] \label{p90} 
    If we take $x_0=0,x_1=1$ and
    \begin{align}
        \label{det10} 
   g_\mu(s)=\mu(0,s), 
    \end{align}
    we investigate if at some point $t \in (0,1]$ the mean function $s \rightarrow \mu(t,s)$ has a maximum deviation larger  than $\Delta$ from the initial mean function $s\rightarrow \mu(0,s)$.
    \item[(2)] \label{p91} 
    With the choice $x_0>0,$ 
       \begin{align}
        \label{det11} 
        {g_\mu} (s)=(1/x_0)\int_{0}^{x_0}\mu(t,s)dt
        \end{align}
        and $x_1=1$, we investigate if at some time point $t \in [x_0,x_1]$ the mean function has a maximum deviation larger  than $\Delta$ from   from its average function $s \rightarrow g_\mu(s)$ calculated on the time interval $ [0,x_0]$.   
\end{enumerate} 
}
\end{example}

In order to estimate the maximum deviation $d_\infty$ we propose to estimate $\mu$  via local linear regression (see \cite{fan1996}) and then use a plug-in estimator for the maximum deviation combined with an appropriate estimator of the threshold function $g_\mu$. As we assume that the functional time series is fully observed only smoothing with respect to the $t$-direction is required at this point. For this purpose  let $K$ denote a kernel function (see  Assumption (A4) below for details) and define $K_h(\cdot) =K(\cdot / h)$ for $h>0$. For fixed $s \in [0,1]$ the local linear estimator $\hat \mu_{h_n}$ of the function $\mu$ at  the point $(t,s)$ with
positive bandwidth $h_n=o(1)$ as $n \rightarrow \infty$ is defined by the first coordinate of the
minimizer
\begin{align}
    \label{p4}
    (\hat \mu_{h_n}(t,s), \hat {\mu^\prime}_{h_n}(t,s))=\underset{c_0,c_1 \in \R}{\arg \min}\sum_{j=1}^n(X_j(s)-c_0-c_1(j/n-t))^2K_{h_n}(j/n-t).
\end{align}
In the following we will use a bias corrected version (see \cite{Schucany1977}) of the local linear estimator given by
\begin{align}
    \label{p5}
    \tilde \mu_{h_n}=2\hat \mu_{h_n/\sqrt{2}}-\hat \mu_{h_n}~.
\end{align}
Furthermore, letting  $\hat g_n$ be an appropriate estimator (see Assumption (A6) below) of the benchmark function $g_\mu$ we can define a natural estimator for $d_\infty$ by
\begin{align}
    \hat d_{\infty,n}:=\sup_{t \in I_n} \sup_{s \in [0,1]}| \tilde \mu_{h_n}(t,s)-\hat g_n(s)|~,
\end{align}
where we calculate the supremum with respect to the variable $t$ on the interval 
$$
I_n=[x_0 \lor h_n, x_1 \land (1-h_n)]
$$ 
to take boundary effects in the estimation of $\mu$ into account. We will then reject $H_0(\Delta)$ for large values of the statistic $\hat d_{\infty,n}$. In order to find suitable critical values we will approximate $\hat d_{\infty,n}$ by a maximum over a suitable partial sum process whose quantiles can be obtained by a resampling method, see Section \ref{sec1.3} for details.

\subsection{Assumptions}
For modeling the dependence in the functional time series model \eqref{p1} we will use the concept of $\beta$-mixing as introduced in \cite{Volkonskii59}. To be precise, assume that all random variables are defined on a probability space $(\Omega, \mathcal{A}, \mathbb{P})$, consider two sigma fields $\mathcal{F}$, $\mathcal{G} \subset \mathcal{A}$ and define
\begin{align}   
    \beta(\mathcal{F},\mathcal{G})=\sup\frac{1}{2}\sum_{i=1}^I\sum_{j=1}^J|\p(F_i\cap G_j)-\p(F_i)\p(G_j)|
\end{align}
where the supremum is taken over all partitions $\{F_1,...,F_I\}$ and $\{G_1,...,G_J\}$ of $\Omega$ such that $F_i \in \mathcal{F}$ and $G_i \in \mathcal{G}$. Next we denote by   $\mathcal{F}_k^{k^\prime}$  the sigma field generated by the error process $\{\epsilon_j\}_{k \leq j \leq k^\prime}$ in model \eqref{p1} and define the $\beta$ mixing coefficients of the sequence $\{\epsilon_j\}_{j \in \N}$ by
\begin{align}   
    \beta(k)=\sup_{l \in \N} \beta(\mathcal{F}_1^l,\mathcal{F}_{l+k}^\infty)
\end{align}
The sequence $\{\epsilon_j\}_{j \in \N}$ is called $\beta$-mixing if  $\beta(k)\rightarrow 0$ as $k \rightarrow \infty$. We can now proceed with stating the assumptions we will make throughout this paper.

\begin{enumerate}
    \item  [{\bf(A1)}]  The error process $\{\epsilon_j\}_{j \in \N}$  has bounded exponential moments, i.e. 
    $$
    \E[\exp(\|\epsilon_1\|_\infty)]\leq K<\infty.
    $$
    \item  [{\bf(A2)}]
          $\{\epsilon_j\}_{j \in \N}$  is $\beta$-mixing, with mixing coefficients $\beta(k)$ satisfying
        \begin{align*}
            \sum_{k=1}^\infty (k+1)^{J/2-1}\beta(k)^{\frac{\delta}{J+\delta}}<\infty~, 
        \end{align*}
        for some even $J\geq 8$ and some $\delta>0$.
    \item [{\bf(A3)}] For some $\alpha>0$ such that $\alpha J >2$ there exists a constant  $C>0$ such that 
    \begin{align}
        \sup_{s,s^\prime \in [0,1],j\in \N}\E \big [|\epsilon_j(s)-\epsilon_j(s^\prime)|^J\big ]^{1/J} \leq C|s-s^\prime|^\alpha
    \end{align}
    \item [{\bf(A4)}] The kernel $K$ is a symmetric and twice differentiable function, supported on the interval $[-1,1]$ and satisfies $\int_{-1}^{1} K(x)dx=1$ as well as $K(0)>0$. Further assume that the bandwidth $h_n$ satisfies $h_n\sim n^{-\gamma_1}$ where $1/7<\gamma_1 \leq 1/5$
    
    \item [{\bf(A5)}] The partial derivative $\frac{\partial^2}{\partial t^2}\mu(t,s)$ of the function $\mu:[0,1]^2\rightarrow \R$ in model \eqref{p3} exists on $[0,1]^2$, is continuous in both arguments and Lipschitz  continuous with respect to the first component, that is
    \begin{align}
        \Big |\frac{\partial^2}{\partial t^2}\mu(t,s)-\frac{\partial^2}{\partial t^2}\mu(t',s)\Big | \leq C|t-t'|~,
    \end{align}    
    where the constant  $C$ does not depend on $s$.
    \item [{\bf(A6)}] The estimator $\hat g_n$ of the benchmark function satisfies
    \begin{align}
        \|\hat g_n-g_\mu\|_\infty = o_\p(n^{-\gamma_2}(nh_n)^{-1/2})
    \end{align}
    for some $\gamma_2>0$.
\end{enumerate}
We stress that in contrast to \cite{buecher21},  we make no stationarity assumptions on the error process in model \eqref{p1}. Assumption (A2) and (A3) control the trade off between the smoothness of the error process and its dependence, in contrast to other works considering relevant change points for functional time series we only require $\beta$ instead of $\phi$-mixing (\cite{dette2020}, \cite{dette2022}, \cite{bastian2023}). Less smoothness requires faster decay of the mixing coefficients. For $\alpha<1/2$ condition (A3) is for instance satisfied by the Brownian motion. Assumption (A1) is made for convenience of presentation. Our results remain true if we instead only assume  the existence of moments  at the cost of making the  proofs, the formulations of the analogs of Assumption (A2), (A3) and the statements of the main results much more complicated. Assumption (A3) is similar but weaker than the conditions used in \cite{dette2020} and \cite{bastian2023} where the classical offline change point problem was considered for $C(T)$ valued data. Functional time series models fulfilling Assumption (A2) include Markov chains and AR processes (see Section 4 in \cite{Lu2022}). Assumption (A4) is a standard assumption on kernels in local linear regression while Assumption (A5) is the functional analogue of Assumption 2.4 in \cite{buecher21}, here we can omit the second part of their assumption as our test is based on resampling instead of an asymptotic limit distribution. Finally assumption (A6) ensures that the error of estimating $g_\mu$ by $\hat g_n$ is negligible compared to the error of estimating $\mu$ by $\tilde \mu_{h_n}$ which is satisfied by many common estimators for $g_\mu(\mu)$. We now discuss this assumption in the context of Example \ref{pex1}

\begin{example}[Continuation of Example \ref{pex1}]  ~~   
{\rm 
    \begin{itemize}
        \item [(1)] Consider the case, where the benchmark function is given by the mean function at the time point $0$, that is $g(s) = \mu (0,s)$. Under Assumptions (A1)-(A5) one can show by arguments similar to those in the proof of Lemma \ref{pl3} in the Appendix that
    \begin{align}
        \|\tilde \mu_{\tilde h_n}(0,s)-\mu(0,s)\|_\infty =O_\p((n\tilde h_n)^{-1/2})+O(\tilde h_n^3+(n\tilde h_n)^{-1})
    \end{align}    
    for any bandwidth $\tilde h_n$ such that $\tilde h_n \rightarrow 0$ and $\tilde h_nn\rightarrow \infty$. Choosing, for instance, $\tilde h_n=h_n^{1+\delta}$ for some $\delta>0$ yields an estimator for $g_\mu(s)=\mu(0,s)$ that satisfies (A6).
    \item [(2)] 
    \label{p95}For benchmarks of the form $g_\mu(s)=1/x_0 \int_0^{x_0}\mu(t,s)dt$ one can define
    \begin{align}
        \bar X_n(s)=\frac{1}{nx_0}\sum_{j=1}^{\floor{x_0n}}X_j(s)
    \end{align}
    for which we have 
    \begin{align}
        \Big \|\bar X_n(s)-{1\over x_0 } \int_0^{x_0}\mu(t,s)dt\Big \|_\infty\leq \Big  \|\frac{1}{nx_0}\sum_{j=1}^{\floor{nx_0}}\epsilon_j(s)\Big  \|_\infty +O(n^{-1}) \leq O_\p(n^{-1/2})
    \end{align}
    by the uniform continuity of $\mu(t,s)$ and Theorem 2.1 from \cite{dette2020} (our assumptions are sufficient for the proof of 
    Theorem 2.1 in this reference).  
    \end{itemize}
}  
    
\end{example}

\subsection{A Bootstrap test for gradual changes}

\label{sec1.3}
As the stochastic expansion of $\hat d_{\infty,n}$ will depend sensitively on certain extremal sets of $\mu(t,s)-g_\mu(s)$ we need some further notation before we can formally state the bootstrap procedure and a corresponding result guaranteeing its validity. To this end we define for any $\rho\geq 0$ the sets
\begin{equation}
\label{p201}
\begin{split}
    \hat{\mathcal{E}}^\pm_{\rho}&=\big\{(t,s) \in I_n\times[0,1] \big|  \pm(\tilde \mu_{h_n}(t,s)-\hat g_n(s)) \pm \rho \geq \|\tilde \mu_{h_n}(\cdot,\cdot)-\hat g_n(\cdot)\|_{\infty,I_n\times[0,1]}\big\}   \\
\mathcal{E}^\pm_{\rho}&=\big\{(t,s) \in I_n\times[0,1] \big|  \pm(\mu(t,s)-g_\mu(s)) \pm \rho \geq \|\mu(\cdot,\cdot)-g_\mu(\cdot)\|_{\infty,I_n\times[0,1]}\big\}   ~, 
\end{split}
\end{equation}
and \begin{align}
    &\mathcal{E}:=\mathcal{E}^+_0\cup \mathcal{E}^-_0 ,~
     \hat{\mathcal{E}}:= \hat{\mathcal{E}}^+_0\cup \hat{\mathcal{E}}^-_0~.
\end{align} 
With these notations we can derive a   stochastic approximation of  the statistic $\hat d_{\infty,n}$ that will be essential for the development of a bootstrap procedure to obtain valid quantiles for the distribution of $\hat d_{\infty,n}$. Defining
$$
d_{\infty,n}=\sup_{(t,s)\in I_n\times [0,1]}|\mu(t,s)-g_\mu(s)|
$$ 
we obtain the following result, which is proved in the appendix.
\begin{theorem}
\label{pt2}
  Let Assumptions (A1)-A(6) be satisfied and denote by $\rho_n$  any sequence such that $\rho_n^{-1}=o\big ((nh_n)^{1/2}h_n^{(3+\alpha^{-1})/J}\big )$. Further assume that $d_\infty>0$. We then have with probability converging to $1$ that
   \begin{align}
   \label{p33}
       \hat d_{\infty,n}-d_{\infty,n} &\leq  \sup_{(t,s) \in \mathcal{E}_{\rho_n}}\Big(\frac{s(t,s)}{nh_n}\sum_{j=1}^n\epsilon_j(s)K^*_{h_n}(j/n-t)\Big)+ o_\p(n^{-\gamma_2}(nh_n)^{-1/2}) \\
       & \leq \hat d_{\infty,n}-d_{\infty,n}+O(\rho_n)~, \nonumber 
   \end{align}   
   where $s(t,s)=\text{sign}(\mu(t,s)-g_\mu(s))$, $\mathcal{E}_{\rho_n}$ is defined in \eqref{p201} and the kernel $K^*$ is given by $K^*=2\sqrt{2}K(\sqrt{2}x)-K(x)$.
\end{theorem}

Before we can proceed defining our bootstrap procedure we need some further notation, let $q=q_n>r=r_n$ be two positive integers with $2(q+r)<n$, and define the sets
\begin{align*}
    &I_1=\{1,...,q\},\\
    &J_1=\{q+1,...,q+r\},\\
    &\quad \quad \quad \vdots \\
    &I_m=\{(m-1)(q+r)+1,...,(m-1)(q+r)+q\},\\
    &J_m=\{(m-1)(q+r)+q+1,...,m(q+r)\},\\
    &J_{m+1} = \{m(q+r),...,n\}
\end{align*}
where $m=m_n=\floor{n/(q+r)}$. In other words we decompose the sequence $\{1,...,n\}$ into large blocks $I_1,...,I_m$ of length $q$ separated by small blocks $J_1,...,J_{m}$ of length $r$ and a remaining part $J_{m+1}$.\\
Additionaly we introduce a grid on the square $I_n\times[0,1]$ with mesh width and length proportional to  $1/n$ and $1/n^{1/\alpha}$, respectively. We denote the set of grid points by $P$.  In particular we have that $|P|\simeq n^{1+1/\alpha}$. Take for any $(t,s)\in I_n\times [0,1]$ the closest  point in $P$ and denote it by $(t_P,s_P)$ (if there are multiple choose the one with minimal $l_1$-norm). Next we define for any $\rho >0$ the sets
\begin{align}
     &\mathcal{E}_{\rho,P}:=\{(t_P,s_P)|(t,s) \in  \mathcal{E}^+_{\rho}\cup \mathcal{E}^-_{\rho}\}\\
     &\hat{\mathcal{E}}_{\rho,P}:=\{(t_P,s_P)|(t,s) \in  \hat{\mathcal{E}}^+_{\rho}\cup \hat{\mathcal{E}}^-_{\rho}\}~.
\end{align} 
Finally, let  $\nu_1,...,\nu_m$ be independent standard normal distributed random variables. We can then define, for any $\rho \geq 0$, a block bootstrap version of the random variable
\begin{align}
    \sup_{(t,s) \in \mathcal{E}_{\rho}}\Big(\frac{ s(t,s)}{\sqrt{nh_n}}\sum_{j=1}^n\epsilon_j(s)K^*_{h_n}(j/n-t)\Big)
\end{align} in \eqref{p33} by
\begin{align}
    \label{p21}
      \sup_{(t,s) \in \hat{\mathcal{E}}_{\rho,P}}\Big(\frac{\hat s(t,s)}{\sqrt{mqh_n}}\sum_{l=1}^m\nu_l\sum_{j \in I_l}\hat \epsilon_j(s)K^*_{h_n}(j/n-t)\Big)
\end{align}
where $\hat \epsilon_j(\cdot)=X_{n,j}(\cdot)-\tilde \mu(j/n,\cdot)$ denotes the residual at time $j/n$ $ (j=1,...n)$, $\hat s(t,s)=\text{sign}(\tilde \mu_{h_n}(t,s)-\hat g_n(s))$  and $\tilde \mu $ is the bias corrected local linear estimator defined by \eqref{p5}  . For the bootstrap to work we additionally also need the following mild assumption:
\begin{itemize}
    \item [\textbf{(A7)}] For each $n\in \mathbb{N} $ there exists  a pair $(t_n,s_n) \in \mathcal{E}_{\rho_n,P}$ such that 
    \begin{align}
        \frac{1}{nh_n}\sum_{l=1}^{m}\E\Big [\Big  (\sum_{j \in I_l}\epsilon_j(s_n)K^*_{h_n}(j/n-t_n))\Big  )^2\Big  ]\gtrsim 1
    \end{align}
\end{itemize}
\begin{remark} ~~ 
\label{pr1}
    \rm 
    \begin{enumerate}
        \item [(1)]
        The above condition essentially states that the average variance of the large bootstrap blocks can not degenerate uniformly over the grid $\mathcal{E}_{\rho_n,P}$.  For illustration consider the case where the error process is second order stationary and $K^*$ is constant and supported on  the interval $[-1,1]$.         
         Let 
         $$J =  \big  \{ \lfloor nt_n  \rfloor - \lfloor \tfrac{ nh_n }{2}  \rfloor  +1, 
         \lfloor nt_n  \rfloor - \lfloor \tfrac{ nh_n}{2}  \rfloor +2, \ldots  , \lfloor nt_n  \rfloor +  \lfloor \tfrac{  nh_n}{2}  \rfloor +1
           \big  \}   $$   
            denote the set of consecutive indices ``centered at $nt_n$  and of size $nh_n$'', then
        \begin{align}
            \frac{1}{nh_n}\sum_{l=1}^{m}\E\Big [\Big  (\sum_{j \in I_l}\epsilon_j(s_n)K^*_{h_n}(j/n-t_n))\Big  )^2\Big  ]&\simeq \frac{1}{mqh_n}\sum_{l=1}^{m}\E\Big [\Big  (\sum_{j \in I_l\cap J}\epsilon_j(s_n)\Big  )^2\Big  ]\\
            & \simeq \frac{1}{mh_n}\sum_{l=1}^m\text{Var}\Big(q^{-1/2}\sum_{j \in I_l\cap J}\epsilon_j (s)\Big)
        \end{align}        
     By construction of $J$   we have, for all but two indices $l \in \{1,...,m\}$, that either $I_l\cap J=I_l$ or $I_l \cap J=\emptyset$.  In particular 
      the number of sets $I_l$ with  $I_l \cap J=I_l$ is proportional to $mh_n$. Using this and  second order stationarity of the error process we have
        \begin{align}
           \frac{1}{nh_n}\sum_{l=1}^{m}\E\Big [\Big  (\sum_{j \in I_l}\epsilon_j(s_n)K^*_{h_n}(j/n-t_n))\Big  )^2\Big  ] \simeq \text{Var}\Big(q^{-1/2}\sum_{j=1}^q\epsilon_j(s)  \Big)
        \end{align}
     Lemma 1 from \cite{Bradley1997} then yields that the condition 
    \begin{align}
        \E\Big [\Big (\sum_{j=1}^n\epsilon_j(s)\Big )^2\Big ]\rightarrow \infty
    \end{align} suffices for (A7) to hold. 
    \item [(2)] There is considerable flexibility when it comes to the choice of grid $P$, in particular its mesh can also be coarser as long as a version of Lemma \ref{pl5} in the Appendix holds. This is particularly interesting when one considers functional data that, instead of being continuously  observed, is available only on a discrete grid with an additional error. To be precise consider, for instance, the situation where one observes data which can be modelled by 
    \begin{align}
    \label{p101}
        Y_{n,ij}=X_{n,j}(s_i)+\sigma(s_i)z_{ij} \quad 1 \leq i \leq N, 1 \leq j \leq n
    \end{align}
  Here $\sigma(\cdot)$ is a positive function on the unit interval, $X_{n,j}$ are the observations in \eqref{p1}, $\{s_1,...,s_N\}\subset [0,1]$ and $\{z_{ij}\}_{1 \leq i \leq N, 1 \leq j \leq n}$ is a collection of centered and independent observations, which is independent of $\{ X_{n,j}(s_i) \}_{1 \leq i \leq N, 1 \leq j \leq n}$. In other words we observe the time series only on a discrete grid subject to some additional random error. Assuming that the latent variables $X_{n,j}$ are sufficiently smooth and  that $\{s_i\}_{1 \leq i \leq N}$ becomes sufficiently dense as $N:=N(n)\rightarrow \infty$ one can adopt our methodology in a straightforward manner and we expect that an analogue to Theorem \ref{pt4} then holds. To be a little more concrete, in this case  one  has to use a grid whose elements are of the form 
    \begin{align}
        (j/n,s_i) \quad 1 \leq i \leq N, 1 \leq j \leq n
    \end{align} to calculate the bootstrap statistic $\eqref{p21}$. Similarly one can simply calculate the supremum $\hat d_{\infty,n}$ over this grid instead of the whole product $I_n \times [0,1]$. We emphasize that there is no need of  further smoothing the observations  in \eqref{p101} to apply our approach.
    \end{enumerate}
\end{remark}

In the following we denote by $\hat q^*_{1-\alpha}$ the $(1-\alpha)$-quantile of the statistic \eqref{p21}. Finally we define our test statistic by
\begin{align}
\label{p100}
    \hat T_{n,\Delta}=\sqrt{nh_n}(\hat d_{\infty,n}-\Delta),
\end{align}
and propose to reject the null hypothesis in \eqref{p3} whenever 
\begin{align}
\label{p92}
    \hat T_{n,\Delta}\geq q^*_{1-\alpha}~.
\end{align}
The main result of this section shows that this decision rule defines a consistent and asymptotic level $\alpha$ test for the hypotheses in \eqref{p3}.

\begin{theorem}
\label{pt4}
   Let Assumptions (A1)-(A7) be satisfied and let $\rho_n$ be any sequence with   $\rho_n^{-1}=o\big ((nh_n)^{1/2}h_n^{(3+1/\alpha )/J}\big )$. Additionally,  assume that for some constant $C>0$  there exists a constant $c>0$ such that
    \begin{align}            
        r/q \leq Cn^{-c}~, ~~
        qh_n^{-1/2}\leq Cn^{1/2-c}~, ~~
        qh_n^{-3-6/J}\leq Cn^{1-c}~, ~~
        m\beta(r)=o(1)~.
    \end{align} Then the following statements are true.
    \begin{itemize}
    \item [(i)] Under $H_0(\Delta)$ we have  
    \begin{align}
    \label{p44}
         \limsup_{n \to \infty}\p(\hat T_{n,\Delta}>q^*_{1-\alpha})\leq \alpha~.
    \end{align}    
    \item [(ii)] Under $H_1(\Delta)$ we have
    \begin{align}
    \label{p45}
        \lim_{n \to \infty}\p(\hat T_{n,\Delta}>q_{1-\alpha}^{*})=1.
    \end{align}
\end{itemize} 
\end{theorem}

\begin{remark} ~~
\label{pr2}
    \rm 
    \begin{enumerate}
        \item [1)]An important question from a practical point of view is the choice of the threshold $\Delta > 0 $, which depends on the specific application under consideration.
    For example, When it comes to climate science, a typical threshold  is the 
    $1.5$ degrees Celsius of the global average surface temperature above pre-industrial temperatures.  On the other hand, if such information is difficult to obtain we 
    can  also determine a threshold from the data which can serve as  measure of evidence  for a  deviation of $\mu$ from  the benchmark function $g_\mu$
   with a controlled type I error  $\alpha$.
   
    To be precise, note that the hypotheses $H_0(\Delta)$ in \eqref{p3} are nested, that the test statistic \eqref{p100} is monotone in $\Delta$ and that the quantile  $q^*_{1-\alpha}$ does not depend on $\Delta$. Consequently,  rejecting $H_0(\Delta)$ for $\Delta=\Delta_1$ also implies rejecting $H_0(\Delta)$ for all $\Delta<\Delta_1$. The sequential rejection principle then yields that we may simultaneously test the hypotheses \eqref{p3} for different choices of $\Delta\geq 0$ until we find the minimum value $\hat \Delta_\alpha$  for which $H_0(\Delta_0)$ is not rejected, that is 
    \begin{align}
        \hat \Delta_\alpha:=\min \big \{\Delta \ge 0 \,| \, T_{n,\Delta}\leq q^*_{1-\alpha} \big  \}=\big(\hat d_{\infty,n}-q^*_{1-\alpha}(nh_n)^{-1/2}\big)\lor 0~.
    \end{align}
    Consequently, one may postpone the selection of $\Delta$ until one has seen the data. 
    \item [2)] A natural question is to ask what happens when choosing $\Delta=0$ in the hypotheses \eqref{p51}. In this case the expansion provided in Theorem \ref{pt2} does not hold true anymore. This is related to the directional Hadamard differentiability of the supremum norm (and its failure to be directionally differentiable in 0, see \cite{Carcamo2020}). It is however possible to define a different bootstrap procedure that is also based on local linear regression and high dimensional Gaussian approximations for this case. As argued in Section 1 our main interest is detecting a relevant change. Choosing  $\Delta=0$ is equivalent to testing  that  the functions $s \to \mu(t,s)$ coincide with the function $s \to  g_\mu(s)$ for all $t$ which is hard to justify in applications. We hence omit a detailed discussion of a development and the investigation of a  bootstrap procedure for the case $\Delta=0$.   
    \end{enumerate}
    
\end{remark}

\section{The first time of a relevant deviation }
\label{sec3}
  \def\theequation{3.\arabic{equation}}	
  \setcounter{equation}{0}
  
In this section we will develop estimators for the time point 
\begin{align}
\label{det201}
    t^*(s,\Delta):=\inf\{t \in [x_0,x_1] \ | \  |\mu(t,s)-g_\mu(s)|\geq \Delta \}~,
\end{align}
where the sequence of functions deviates for the first time by more than $\Delta$ from the benchmark $g_\mu$ at a given argument $s \in [0,1]$. Moreover we also construct an estimator for the first time point
\begin{align} \label{det200}
    t^*(\Delta)  :=\inf_{s \in [0,1]}t^*(s,\Delta)=\inf\{ t \in [x_0,x_1]  \ | \ \|\mu(t,\cdot)-g_\mu(\cdot)\|_\infty \geq \Delta \} , 
\end{align}
where there is a deviation by more than $\Delta$ at at least  one point in the interval $[x_0,x_1]$. Here we use the convention that $\inf \emptyset = \infty$. First note that the continuity of the function $d(t,s):=\mu(t,s)-g_\mu(s)$ yields the representation
\begin{align}
\label{p60}
    t^*(s,\Delta)=x_0+\int_{x_0}^{x_1}\1\Big (\max_{t \in [x_0,y]}|d(t,s)|<\Delta\Big ) dy+\infty \cdot \1 (d_\infty <\Delta )
\end{align}
which naturally leads to the estimator 
\begin{align}
\label{p61}
    \hat t^*(s,\Delta)=&x_0\lor h_n +\int_{I_n}\1\Big (\max_{t \in [x_0 \lor h_n,y]}|\hat d(t,s)|<\Delta-\delta_n\Big )dy \\
     & \quad  +\infty \cdot \1  (\hat d_{\infty,n} <\Delta-\delta_n )
\end{align}
where $\hat d(t,s)=\tilde \mu_{h_n}(t,s)-\hat g_n(s)$ and $\delta_n$ is a sequence converging to 0.

As the relation of the integrated indicators in \eqref{p60} and \eqref{p61} depends on the smoothness of $d(t,s)$ in $t^*(\Delta)$ we define a local modulus of continuity of $d(t^*(\Delta),s)$ to capture its degree of smoothness. To be precise for a fixed $s \in [0,1]$ we call an increasing function $\omega(\cdot,s)$ a local modulus of continuity of $d(\cdot,s)$ at $t^*(\Delta)$ if
\begin{align}
    |d(t^*(\Delta),s)-d(t,s)|\leq \omega(|t^*(\Delta)-t|,s)
\end{align}
for all $s \in [0,1]$ and $t$ with $|t^*(\Delta)-t|<\delta_n/2$. Note that such a modulus always exists by the uniform continuity of $d(t,s)$. Now we can formally state the properties of $\hat t^*(s,\Delta)$.

\begin{theorem}
\label{pt5}
    Let Assumptions (A1)-(A6) be satisfied, $\delta_n$ be such that \\ $\delta_n^{-1}=o\big ((nh_n)^{1/2}h_n^{(3+\alpha^{-1})/J}\big )$ and, for a fixed $s \in [0,1]$, let $\omega(\cdot,s)$ be a local modulus of continuity of $d(\cdot,s)$ at the point $t^*(\Delta)$. Defining 
    \begin{align}
        S=\big \{s \in [0,1] \Big| t^*(s,\Delta)<\infty\big \}
    \end{align} we have
    \begin{align}
    \label{p202}
        &\p\Big(\forall s \in S : \hat t^*(s,\Delta) \geq t^*(\Delta,s)-\omega^{-1}(\delta_n/2,s)\lor h_n \Big)=1-o(1)\\
        &\p\Big(\forall s \in S : \hat t^*(s,\Delta) \leq t^*(\Delta,s)+ h_n \Big)=1-o(1)        ~.
    \end{align}
    In particular, it holds that
    \begin{align}
    \label{p65}
        \norm{\hat t^*(\cdot,\Delta)-t^*(\cdot,\Delta)}_{\infty,S}=o_\p(1)~, 
    \end{align}
    and for all $s \in [0,1]\setminus S$ we have 
    \begin{align}
    \label{p64}
        \p(\hat t^*(s,\Delta)<\infty)=o(1)~.
    \end{align}
    In the case that $S=\emptyset$ we can even extend this to
    \begin{align}
        \p(\exists s : \hat t^*(s,\Delta)<\infty)=o(1)~.
    \end{align}
\end{theorem}

\begin{remark}
    \rm   Equation \eqref{p65} is a direct consequence of the estimates \eqref{p202} as we can always choose $\omega(\cdot,s)$ to be independent of $s$ due to the uniform continuity of the function $d(t,s)$. The above statement of Theorem \eqref{pt5} merely allows for sharper estimates in   cases where  $d(t,s)$ is particularly smooth at $t^*(\Delta)$ for certain choices of $s$. Further we note that \eqref{p65} implies that $t^*(\Delta)=\inf_{s \in [0,1]}t^*(s,\Delta)$ is consistently estimated by $\hat t^*(\Delta)=\inf_{s \in [0,1]}\hat t^*(s,\Delta)$.       
  
\end{remark}

\section{Finite sample properties}
\label{sec4}
  \def\theequation{4.\arabic{equation}}	
  \setcounter{equation}{0}
In this section we will investigate the finite sample properties of the proposed procedure by means of a simulation study. We also illustrate the new methods by an application to a real data set. 

\subsection{Simulation study}
We consider two different choices for $\mu$ in model \eqref{p1} which are given by 

\begin{align}
    \label{p80} 
    &\mu_1(t,s)=\begin{cases}
        s(1-s) \quad \ &0 \leq t\leq \frac{1}{8}\\
        s(1-s)+2\sin(\pi(t-1/8)) \quad \ & \frac{1}{8}  \leq t \leq \frac{5}{8}\\
        s(1-s)+2\sin(\pi(t-1/8))-2s(1-s)(t-5/8) \quad \ & \frac{5}{8} < t\leq 1
    \end{cases} ~~~~~~~\\
    &\mu_2(t,s) =\begin{cases}
        4+f(s)+t(1-t)   \qquad & 0 \leq t\leq \frac{1}{4} \\
        4+f(s)+t(1-t)+s^2(t-1/4)^2  \quad  & \frac{1}{4}\leq t\leq 1 
    \end{cases}   
    \label{p80a}
\end{align}
where the function $f$ in the definition of $\mu_2$ is given by 
\begin{align*}
    f(t)= \Big [ {1+\Big(\frac{1-t}{t}\Big)^2} 
    \Big ]^{-1} \quad 0<t\leq 1~.
\end{align*}
For the function $\mu_1$ the  corresponding benchmark function is chosen as $g_{\mu_1} (s) = \mu_1 (0,s)$ (see  Example \ref{pex1}(i)). For the function $\mu_2$ we use the benchmark function in Example \ref{pex1}(ii) where  we choose $x_0=1/4$, that is 
\begin{align}
    g_{\mu_2}(s)=4\int_{0}^{1/4}\mu(t,s)dt~.
\end{align}
For both choices of $\mu$ we consider two different error processes in model \eqref{p1}, that is 
\begin{align}
\label{p81}
    &\epsilon_j=\frac{1}{2}B_j\\
    \label{p82} 
    &\epsilon_j=\frac{1}{\sqrt{5}}(B_j+\frac{1}{2}B_{j-1})~,
\end{align}
where   $B_j=\{B_j(s)\}_{s \in [0,1]}, 1 \leq j \leq n$ are independent Brownian bridges. 

To choose the bandwidth parameter $h_n$ for the calculation of the estimator $\tilde \mu_{h_n}$ we perform $k-$fold cross validation with $k=10$, see for instance \cite{Hastie2009} page 242. To be precise we use we use  the following algorithm:

\begin{enumerate}
    \item[\textbf{(1)}] Split the data into $k=10$ sets  $S_1,...,S_{10}$ of the same length
    \item[\textbf{(2)}] For $h_n=1/n$ and each set $S_i$ calculate the estimator $\tilde \mu_{h_n}^{(i)} $ based on the data in the remaining sets
    \item[\textbf{(3)}] Compute the prediction error \begin{align}
        \text{MSE}_{h_n}=\frac{1}{1-h_n/2}\sum_{i=1}^{10}\sum_{j \in S_i}\norm{X_{n,j}(\cdot)-\tilde \mu_{h_n}^{(i)}(j/n,\cdot)}_2^2
    \end{align}
    \item[\textbf{(4)}] Repeat steps (2) and (3) for $h_n=2/3n^{-1/5},...,n^{-1/5}$
    \item[\textbf{(5)}] Choose the bandwidth $h_n$ that minimizes $\text{MSE}_{h_n}$.
\end{enumerate}

The big block length parameter $q$ is selected by the method proposed in \cite{Rice2017} with initial bandwidth $q=\floor{n^{1/5}}$ which requires the choice of a weight function. We have implemented their algorithm for the Bartlett, Parzen, Tukey-Hanning and Quadratic spectral kernel and observed that the quadratic spectral weight function yields the most stable results (these findings are not displayed for the sake of brevity). The small block length parameter is chosen as $r=\ceil{n^{1/10}}$. We further choose $\rho_n=0.1\log(n)/\sqrt{nh_n}$ which matches a similar choice for the estimation of extremal sets in 
 \cite{buecher21}.

\begin{figure}[H]
    \centering

    \includegraphics[scale= 0.35]{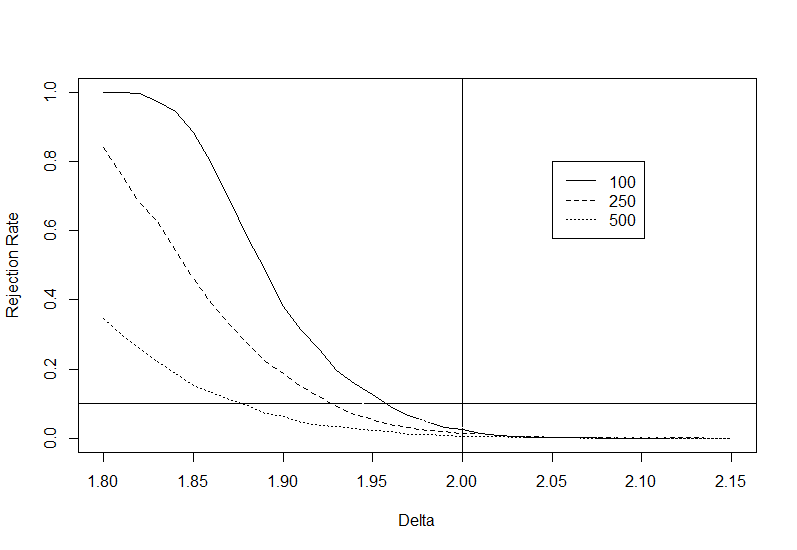}  
    \includegraphics[scale= 0.35]{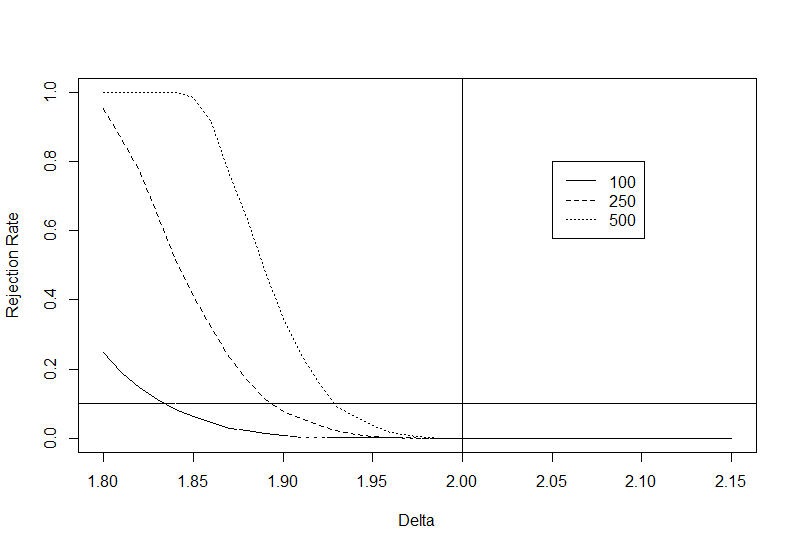}  
    
    \caption{\it    
     Empirical rejection probabilities of the test \eqref{p92} for the hypotheses \eqref{p3} for different  $\Delta$ and  sample sizes $n=100, 250, 500$. The mean function is given by \eqref{p80} and the error processes by \eqref{p81} (left) and \eqref{p82} (right).
     The benchmark function     $g_\mu$ is given by \eqref{det10} and  $d_\infty=2$ in all cases. 
 } 
    \label{Fig:3}
\end{figure}

\begin{figure}[H]
    \centering

    \includegraphics[scale= 0.35]{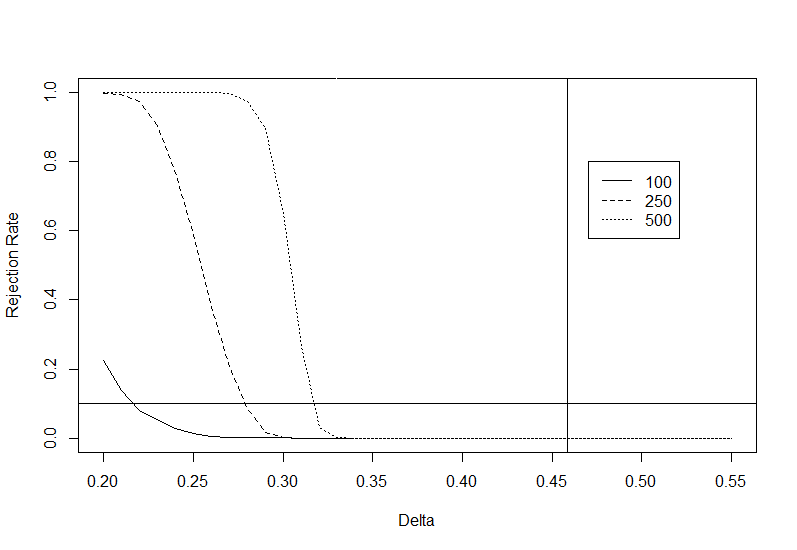}  
    \includegraphics[scale= 0.35]{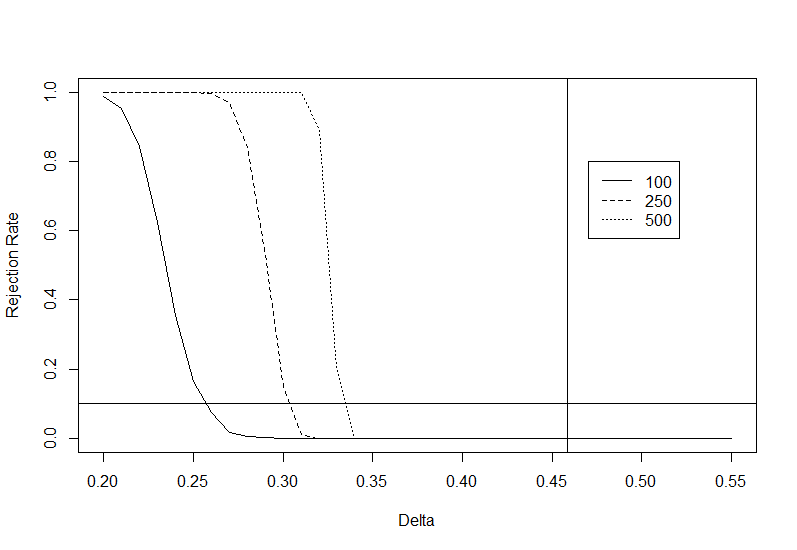}  
    
    \caption{\it  
       Empirical rejection probabilities of the test \eqref{p92} for the hypotheses \eqref{p3} for different  $\Delta$ and sample sizes $n=100, 250, 500$. The mean function is given by \eqref{p80a} and the error processes by \eqref{p81} (left) and \eqref{p82} (right).
     The benchmark function     $g_\mu$ is given by \eqref{det11} with $x_0=0.25$ and  $d_\infty=0.4585$ in all cases.
 } 
    \label{Fig:4}
\end{figure}
The results of this study are presented in Figures \ref{Fig:3} and \ref{Fig:4}, which show the
empirical rejection probabilities of the test \eqref{p92} for the hypotheses \eqref{p3}  with nominal level $\alpha=0.1$  and varying thresholds $\Delta$.
The vertical line corresponds to the true value $d_\infty $, which is 
$
d_\infty = \max_{s,t\in [0,1] } | \mu_1 (t,s) - \mu_1 (0,s)| \approx 2
$
for the function $\mu_1 $ and 
$
    d_\infty = \max_{s,t \in [0,1]}\big|\mu_2(t,s)-4\int_0^{1/4}\mu_2(t,s)dt\big| \approx 0.4585 $
    for the function $\mu_2$.
The results are obtained  by $1000$ simualtion runs, where we performed  $200$ bootstrap repetitions in each run to calculate the critical value of the test.

We observe that the rejection probabilities are decreasing with an increasing threshold $\Delta$, which reflects the fact that it is easier to detect
deviations for smaller thresholds. 
Note that small values of $\Delta$ correspond to the alternative and that  the power of the test quickly increases as $\Delta $ is decreasing. On the other hand, values of $\Delta \geq d_\infty$ correspond to the null hypothesis and the test keeps its nominal level. These results  confirm the asymptotic  theory in  Theorem \ref{pt4}.
A comparison of the results for the independent case (left part of Figure \ref{Fig:3}) and the dependent case (right part of Figure \ref{Fig:3}) shows  that in the setting \eqref{det10} the test \eqref{p92} is more conservative for dependent data. On the other hand,  in this case the rejection rates increases more quickly with increasing distance to the ``boundary'': $d_\infty = \Delta$ of the  hypotheses. When $g_\mu$ is instead given by \eqref{det11}  the test is also conservative, but different from the case
 \eqref{det10} it is generally more powerful for dependent data and its rejection rate also increases more sharply with increasing distance to the ``boundary'': $d_\infty = \Delta$ after it exceeds the nominal level.

\begin{remark} {\rm 
    The bootstrap defined in equation \eqref{p21} 
    does not use a portion of the data by breaking it into small and large blocks. As an alternative, one may use a more standard block bootstrap procedure of the form
    \begin{align}
    \label{p93}
        \sup_{ (t,s) \in \hat{\mathcal{E}}_{\rho,P}}\frac{\hat s(t,s)}{\sqrt{nh_n}}\sum_{j=1}^{n-l+1}\nu_j Y_j(s,t)
    \end{align}
    where 
    
$Y_j(s,t)=l^{-1/2}\sum_{k=j}^{j+l-1}\hat \epsilon_k(s)K^*_{h_n}(j/n-t)$ which uses all of the available data. Showing that this bootstrap is theoretically valid poses some further technical challenges which are beyond the scope of the present paper. 
 \\   
Moreover, it is not clear if this approach  yields an improvement of the finite sample properties of the test proposed in this paper.  To illustrate this fact we include some simulations with the mean function $\mu$  in \eqref{p80} and the error processes given by \eqref{p81} and \eqref{p82}. The empirical rejection probabilities displayed  in Figure \eqref{Fig:5} and are directly comparable with the results for the test \eqref{p92}, which are shown    in Figure \ref{Fig:3}. 
We observe  that the qualitative behavior of the bootstrap test is similar to the test \eqref{p92}.
However, the alternative test 
is more conservative and as consequence we observe  a lower power compared to the test \eqref{p92}.

    \begin{figure}[H]
    \centering

    \includegraphics[scale= 0.35]{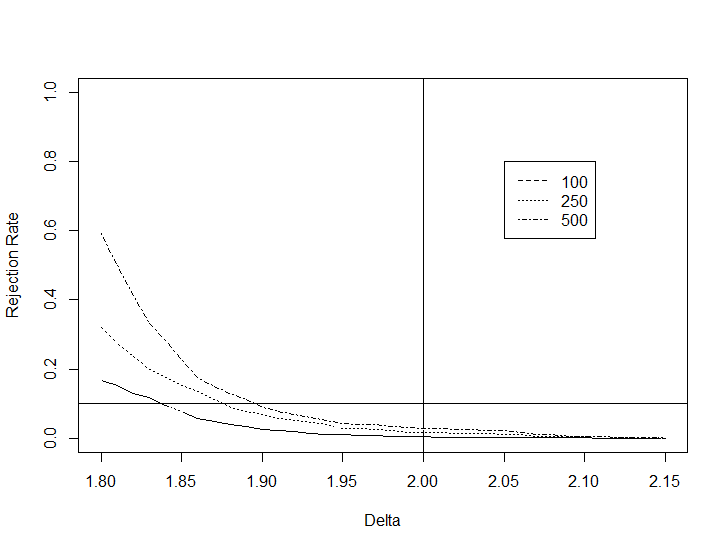}  
    \includegraphics[scale= 0.35]{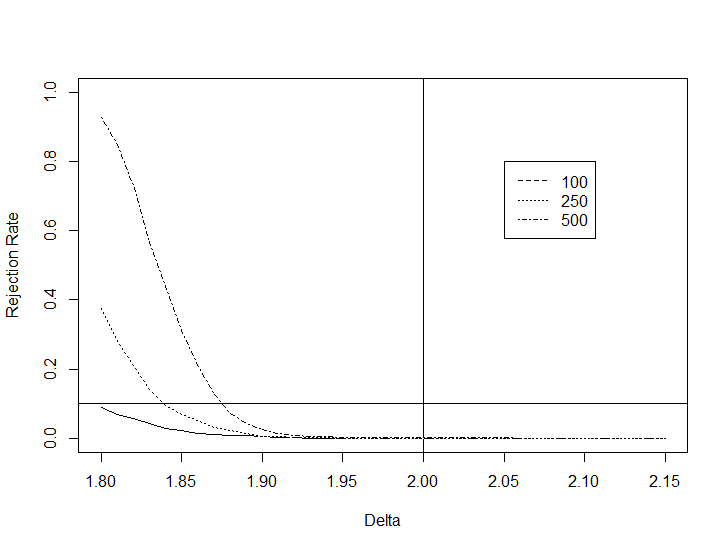}  
    
    \caption{\it    
     Empirical rejection probabilities of the bootstrap test based on \eqref{p93}  for the hypotheses \eqref{p3} for different  $\Delta$ and  sample sizes $n=100, 250, 500$. The mean function is given by \eqref{p80} and the error processes by \eqref{p81} (left) and \eqref{p82} (right).
     The benchmark function     $g_\mu$ is given by \eqref{det10} and  $d_\infty=2$ in all cases. 
 } 
    \label{Fig:5}
\end{figure}
}
\end{remark}

\subsection{Real Data Example}

We applied the proposed methodology to annual temperature data from six weather stations in Canada (the data is available at www.bom.gov.au) which has also been considered in \cite{Aue2017} and \cite{buecher21}. We considered the daily minimal temperatures over a time span of roughly $100$-$150$ years, depending on the weather station. Exemplarily we display a histogram type plot of the data and its estimated mean function $\hat \mu_{h_n}$ for the weather station in Boulia in Figure \ref{Fig:1}. Here we excluded all years for which more than $10\%$ of the observations are  missing and replaced in the other cases  the missing observations by a spline interpolation.
In total $24$  years have been discarded due to a significant lack of observations and the resulting sample size is   $n=112$. 

\begin{figure}[h]
    \centering
    \includegraphics[scale= 0.4]{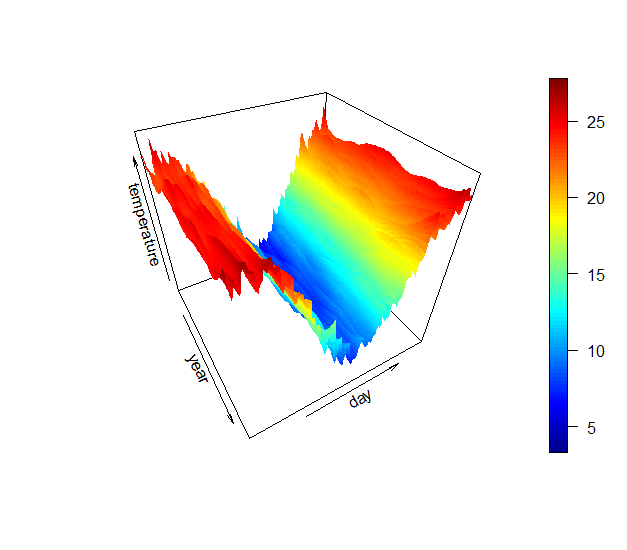}
    \includegraphics[scale= 0.4]{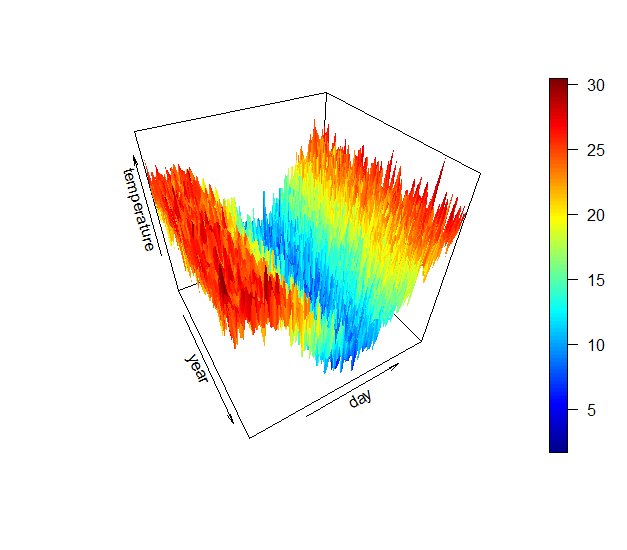}
    
    \caption{\it    
    Minimal temperature measured by a weather station in Boulia between 1888 and 2013.  The raw data is displayed on the right side while the local linear mean regression estimate is displayed on the left side.
} 
    \label{Fig:1}
\end{figure}

Even if the  boundary effects appearing in the  estimation of the  mean function $ t \to \mu (t,\cdot) $ are excluded, we observe that this type of  data cannot be 
well described by a model with a constant mean function $\mu (t,s) = \mu (s)$ or a piecewise constant mean function as discussed  in \cite{Aue2017} and \cite{bastian2023}. In fact it is apparent from  the figure that the functions $ \{ s \to \mu (t,s) \}_t$ are continuously changing with $t$ such that the model \eqref{p1} is more appropriate for this data.

In order to test for a significant deviation from  a benchmark function, we
 normalized (for each weather station) the data to be defined on the unit interval and applied the proposed methodology with 
\begin{align}
\label{det12}
    g_\mu(s)=\frac{1}{x_{0}}\int_0^{x_{0}}\mu(t,s)dt
\end{align}
where $x_{0}$ is chosen such that testing the hypotheses in \eqref{p3} corresponds to testing for a relevant deviation of the temperature in the second half of the 20th century from its historical average  up to the year $1950$.  To estimate $g_\mu$ we use the estimator suggested in \eqref{p95}. The parameters $h,q, \rho_n$ and $r$ are chosen in the same way as in the previous subsection.  We perform  $1000$ bootstrap repetitions to estimate the quantiles  $q^*_{1-\alpha}$ for the test \eqref{p92}.

In  the left part of  Table \ref{Tab:1}  we record for each weather station the largest  threshold $\hat \Delta_{0.1}$ for which the test \eqref{p92} rejects $H_0(\Delta)$ at nominal  level $\alpha=0.1$, see the discussion in Remark \ref{pr2}. For example, for the data from  the station Boulia the null hypothesis in \eqref{p3} is rejected for $\Delta =1.999$ (and all smaller thresholds) with nominal level $\alpha=0.1$.  The null hypothesis   cannot be rejected if $\Delta > 1.999$. The first year where this deviation is detected is the year $1967$.
We also include the estimates $ \hat t^*(\Delta)$  in \eqref{det200} for different choices of $\Delta$.
Estimators where the test \eqref{p92} also rejects the null hypothesis in \eqref{p3} are marked in  boldface. For example, a significant deviation by more than $1.5$ degrees Celsius from the pre-industrial temperatures is detected by the new method for the first time in the year $1956$.

For the sake of comparison we  also include the estimates $\hat t^*_{\rm uni} (\Delta)$ from Table 3 in  \cite{buecher21} which refer to a relevant deviation of more that $\Delta$ degree Celsius from the  pre-industrial time 
in  a univariate time series of the average temperatures in  the month July.     For example, at the station  Boulia the first years of a deviation of the  magnitude
$\Delta=0.5 $ and $1$ occurs in the years  $1957$ and  $1960$, respectively
and - in contrast to the functional approach presented in this paper -  the method 
of \cite{buecher21} does not detect  a deviation by more that $1.5$ degree Celsius from the univariate data. 

\begin{table}[h]
    \centering
\begin{tabular}{|c|c|c|c|c|c||c|c|c|}
\hline
& \multicolumn{5}{|c||}{our methods} & \multicolumn{3}{c|}{\cite{buecher21}} \\
\hline
     $\alpha = 0.1$ & $\hat \Delta_{0.1}$& $\hat t^*(\hat \Delta_{0.1})$ & $\hat t^*(0.5)$ & $\hat t^*(1)$ &$\hat t^*(1.5)$ &  $\hat t^*_{\rm uni}(0.5)$ & $\hat t^*_{\rm uni}(1)$ &$\hat t^*_{\rm uni}(1.5)$  \\
          \hline 
     Boulia & 1.999 &  \textbf{1967} & \textbf{1951} & \textbf{1951} &\textbf{1956} & \textbf{1957} & \textbf{1960} & $\infty$ \\ 
     Otway & 0.558 & \textbf{1951}  & \textbf{1951} &1965&$\infty$  & $\infty$ & $\infty$ & $\infty$\\
     Gayndah & 2.590 & \textbf{1969} & \textbf{1951} &\textbf{1951}&\textbf{1952} & \textbf{1951} &\textbf{1969} & \textbf{1974}\\ 
     Hobart & 0.745 & \textbf{1952} & \textbf{1951}&1959&1975& \textbf{1975}& $\infty$& $\infty$\\
     Melbourne & 1.410& \textbf{1956} & \textbf{1951}&\textbf{1951}&1958  & \textbf{1968} & \textbf{1978} & $\infty$\\
     Sydney & 1.164 &  \textbf{1961} & \textbf{1951}&\textbf{1955}&1970 & \textbf{1978} & $\infty$ & $\infty$\\
       \hline         
\end{tabular}
\caption{\it Second column: maximal threshold $\hat \Delta_{0.1}$ for which the null hypothesis in \eqref{p3} is  rejected  by the test \eqref{p92}. The benchmark function is  is given by the average minimal temperature up to the year $1950$, see \eqref{det12}. 
Column $3$-$6$:  the  estimate $\hat t^* (\Delta) $ defined in \eqref{p61} of  the time of the first relevant deviation where  $\Delta= \hat \Delta_{0.1}$ and $\Delta \in \{0.5,1,1.5\}$.
Column $7$-$9$:  Corresponding estimates  from a  univariate time series of the average temperatures in  the month July taken from  Table 3 in \cite{buecher21}.}
\label{Tab:1}
\end{table}

Comparing these results with those in Table 3 in \cite{buecher21} (see the right part of  Table \ref{Tab:1}) we note that the changes  detected by their univariate  method 
are always detected by the functional approach in this paper as well. Moreover, the new methods detects additional changes and in many cases it  detects the changes substantially earlier. These observations are a consequence of the fact that our functional approach does not require  summarizing the functional data  to a univariate time series, thereby avoiding a loss of power against alternatives where the increase in temperature is not spread out evenly over the whole year.

From a practical point of view one might conversely be concerned about detecting anomalies with a short duration instead of structural changes (for instance one might observe a sequence of particularly warm Januaries while the overall climate is stable during that time). We can address this problem by using the estimator  $ s \rightarrow \hat t^*(s,\Delta)$ 
in \eqref{det201}
instead of  $\hat t^*(\Delta)$  to obtain  a more detailed  understanding  of the first times of 
a relevant change, sidestepping the issue of possibly uninteresting short lived anomalies to some degree. To illustrate this we  discuss the Boulia weather station data. In Figure \ref{Fig:2} we  display  the functions $ s \rightarrow \hat t^*(s,\Delta)$  for multiple choices of $\Delta$. 
Here one can observe that for some parts of the year the estimate for the time of the first relevant deviation is dated substantially later than for other parts, indicating that different mechanism might be responsible for the increase in temperature, respectively. Here our methodology facilitates a very nuanced discussion of the changes in the data, providing an interesting alternative to the discussions in \cite{Aue2017} and \cite{buecher21}.

\begin{figure}[H]
    \centering
    \includegraphics[scale= 0.37]{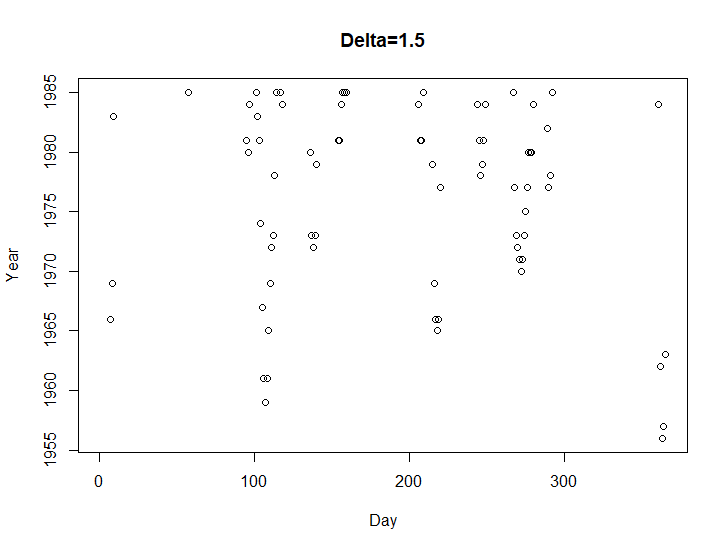}
    \includegraphics[scale= 0.37]{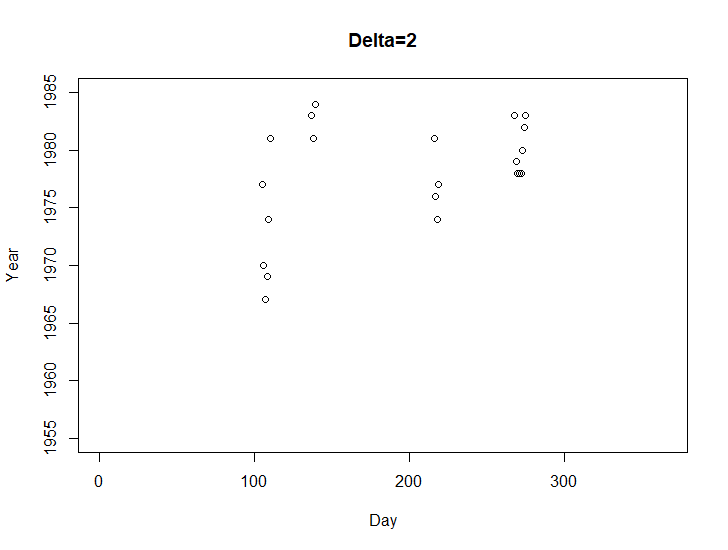}
    \caption{\it Plots of the estimator $\hat t^*(s,\Delta)$ defined in \eqref{p60} for the Boulia weather station data for different choices of $\Delta$.}
    \label{Fig:2}
\end{figure}

\section{Proofs}
\label{sec5}
  \def\theequation{5.\arabic{equation}}	
  \setcounter{equation}{0}
Throughout this section we will assume that Assumption (A1)-(A6) are satisfied. We also introduce some further terminology. 

We call a mean zero real valued random variable $X$ subgaussian with variance proxy $C^2$ if
\begin{align}
    \E[e^{tX}]\leq e^{-\frac{t^2}{2C^2}}~,
\end{align}
for all $t \in \R$.
For a given set $\mathcal{A}$ and a sigma field $\mathcal{Z}$ we call a real valued random variable $X$ subgaussian on $\mathcal{A}$ conditional on $\mathcal{Z}$ with variance proxy $C^2$ if
\begin{align}
     \E[e^{tX}|\mathcal{Z}](\omega) &\leq e^{-\frac{t^2}{2C^2}} \quad \forall \omega \in \mathcal{A}\\
     \E[X|\mathcal{Z}](\omega)&=0  \quad \forall \omega \in \mathcal{A}~,
\end{align}
for all $t \in \R$. 
Further  for two real valued sequences $(a_n)_{n \in \N}, (b_n)_{n \in \N}$ we use the notation  $ a_n \lesssim b_n$ if $a_n \leq Cb_n$ for some constant $C$ that does not depend on $n$ and that may change from line to line. In addition we write $a_n \simeq b_n$ if $a_n \lesssim b_n$ and $b_n \lesssim a_n$. 
We say that an event $\mathcal{A}$ holds with high probability if $\p(\mathcal{A})=1-o(1)$. Finally, 
to make expressions more concise, we also define the process 
\begin{align}
\label{p20}
    \big \{\mathcal{H}_n(t,s)\big \}_{(t,s) \in I_n \times [0,1]}:= \Big\{\frac{1}{\sqrt{nh_n}}\sum_{j=1}^n\epsilon_j(s)K^*_{h_n}(j/n-t) \Big \}_{(t,s) \in I_n \times [0,1]}~.
\end{align}

\subsection{Proof of Theorem \ref{pt2}}
The Theorem follows immediately by combining Lemmas \ref{pl3} and \ref{pl4}.

\begin{Lemma}
\label{pl3}
   We have that
    \begin{align}
        \sup_{ [t,s] \in I_n\times[0,1]}\Big|\tilde \mu_{h_n}(t,s)-\mu(t,s)-\frac{1}{nh_n}\sum_{j=1}^n\epsilon_j(s)K^*_{h_n}(j/n-t)\Big|= O(h_n^3+(nh_n)^{-1}) ,
    \end{align}
     where $K^*=2\sqrt{2}K(\sqrt{2}x)-K(x)$.
\end{Lemma}

\begin{proof}

    Using (A5) combined with the same arguments as in the proof of Lemma B.1 from \cite{dette2015} yields the desired statement where the uniformity conditions in (A5) ensure that the errors in the Taylor expansions can also be bounded uniformly in $s$.

\end{proof}

\begin{Lemma}
\label{pl4}
   Let $\rho_n$ be any sequence such that $\rho_n^{-1}=o\big ((nh_n)^{1/2}h_n^{(3+\alpha^{-1})/J}\big )$. Assume that $d_\infty>0$. We then have that with probability tending to one that 
   \begin{align}
       \hat d_{\infty,n}-d_{\infty,n} &\leq \sup_{(t,s) \in \hat{\mathcal{E}}_{\rho_n}}\hat s(t,s)(\tilde \mu_{h_n}(t,s)-\mu(t,s)))+o_\p(n^{-\gamma_2}(nh_n)^{-1/2})\\
       & \leq \hat d_{\infty,n}-d_{\infty,n} +O_\p(\rho_n) 
   \end{align}
   The same also holds true when $\hat{\mathcal{E}}_{\rho_n}$ and $\hat s(t,s)$ are replaced by $\mathcal{E}_{\rho_n}$ and  $s(t,s)$, respectively.
\end{Lemma}
\begin{proof}
    We only show the inequality involving the set $\hat{\mathcal{E}}_{\rho_n}$, the corresponding inequality for $\mathcal{E}_{\rho_n}$ follows by exactly the same arguments. Also note that Lemma \ref{pl2} yields that we may assume that $\hat s(t,s)=s(t,s)$ in the following arguments. \\
    
    First we define the quantities
    \begin{align}
        \Delta^+_{\rho_n}&:= \sup_{(t,s) \in \hat{\mathcal{E}}^+_{\rho_n}}(\tilde \mu_{h_n}(t,s)-\hat g_n(s)-(\mu(t,s)-g_\mu(s))),\\
        \Delta^-_{\rho_n}&:= \sup_{(t,s) \in \hat{\mathcal{E}}^-_{\rho_n}}-(\tilde \mu_{h_n}(t,s)-\hat g_n(s)-(\mu(t,s)-g_\mu(s)),\\
        \Delta_{\rho_n}&:=\max\{ \Delta^+_{\rho_n},\Delta^-_{\rho_n}\}.
    \end{align}

    Note that 
    \begin{align}
    \hat d_{\infty,n}-d_{\infty,n} \leq  \Delta_{\rho_n}  
    \end{align}
    with high probability, 
    due to the fact that, by Lemma \ref{pl2}, we have $\mathcal{E}^{\pm} \subset \hat{\mathcal{E}}^\pm_{\rho_n}$ with high probability. Combining this observation with Assumption (A6) then yields the first inequality in the statement of the lemma. 
    For the second inequality we choose $(t_n,s_n) \in \hat{\mathcal{E}}^+_{\rho_n}$ such that 
    \begin{align}
      d_{\infty,n}- 3\rho_n &\leq   \mu(t_n,s_n)-g_\mu(t_n,s_n)~,\\
        \Delta^+_{\rho_n} - 1/n &\leq \tilde \mu_{h_n}(t_n,s_n)-\hat g_n(s_n)-(\mu(t_n,s_n)-g_\mu(s_n))~, 
    \end{align}
    where we the use Lemma \ref{pl2} to obtain the first inequality for any $(t,s) \in \hat{\mathcal{E}}^+_{\rho_n}$. From these inequalities we obtain  (with high probability)
    \begin{align}
        \Delta^+_{\rho_n} &\leq  \tilde \mu_{h_n}(t_n,s_n)-\hat g_n(s_n)-(\mu(t_n,s_n)-g_\mu(s_n)) + 1/n\\
        & \leq \sup_{(t,s) \in \hat{\mathcal{E}}^+_{\rho_n}}  (\tilde \mu_{h_n}(t_n,s_n)-\hat g_n(s_n))  - d_{\infty,n}+1/n+3\rho_n~\\
        & \leq \hat d_{\infty,n}  - d_{\infty,n}+1/n+3\rho_n~,
    \end{align}
    where we again used that $\mathcal{E} \subset \hat{\mathcal{E}}_{\rho_n}$ holds with high probability to obtain the last line. A similar argument yields an analogous inequality for $\Delta^-_{\rho_n}$. Combining these two inequalities yields the second inequality  in the statement of the lemma by assumption (A6).

\end{proof}

\begin{Lemma}
\label{pl1}
    Let $\alpha$  denote the constant from Assumption (A3) and define
    \begin{align}
        d((t,s),(t',s')):=|s-s^\prime|^\alpha~+|t-t'|h_n^{-1}~,
    \end{align} then we have for any $\tau>0$ that
    \begin{align}
       \sup_{d((t,s),(t',s'))<\tau}|\mathcal{H}_n(t,s)-\mathcal{H}_n(t',s')| = O_\p(\tau h_n^{-1-(3+\alpha^{-1})/J})
    \end{align}
    
\end{Lemma}
\begin{proof}
   Note that
    \begin{align}
        \E[|\mathcal{H}_n(t,s)-\mathcal{H}_n(t^\prime, s^\prime)|^J]^{1/J}\leq K_1\Big( |s-s^\prime|^\alpha~+|t-t'|h_n^{-1}\Big)=K_1d((t,s),(t',s')~,
    \end{align}
    which follows by Assumptions (A2) and (A3) and the arguments used for the proof of Theorem 3 from \cite{Yoshi78}. Therefore Theorem 2.2.4 in \cite{Wellner1996} yields that for any $\nu, \tau>0$ and some $K_2>0$ depending only on $J$ and $K_1$ we have 
    \begin{align*}
        \E\Big [\sup_{d((t,s),(t',s'))<\tau}|\mathcal{H}_n(t,s)-\mathcal{H}_n(t^\prime, s^\prime)|^J\Big ]^{1/J}
        &\lesssim K_2 \big( h_n^{-1/J}\int_0^\nu \epsilon^{-\frac{1+1/\alpha}{J}}d\epsilon+\tau h_n^{-2/J}\nu^{-\frac{2+2/\alpha}{J}} \big ) \\
        &\lesssim h_n^{-1/J}\nu^{1-\frac{1+1/\alpha}{J}}+\tau h_n^{-2/J} \nu^{-\frac{2+2/\alpha}{J}}~.\\
        &\lesssim h_n^{-1/J}\nu^{1/2-1/J}+\tau h_n^{-2/J} \nu^{-2/J-1}
    \end{align*}
    If we choose $\nu=h_n$ the assertion follows by an application of the Markov inequality.
    
\end{proof}

\begin{Lemma}
\label{pl2}
    We have that
    \begin{align}
        &\|\tilde \mu_{h_n}(t,s)-\hat g_n(t,s)-(\mu(t,s)-g_\mu(t,s))\|_{\infty,I_n \times [0,1]}=O_\p \big ((nh_n)^{-1/2}h_n^{-(3+\alpha^{-1})/J} \big )~,
    \end{align}
    which also implies that 
    $$\mathcal{E}^+_0 \subset \hat{\mathcal{E}}^+_{\rho_n}~\text{  and
     } ~\hat{\mathcal{E}}^+_0 \subset \mathcal{E}^+_{\rho_n}
     $$ 
     with high probability for any sequence $\rho_n$ such that $\rho_n^{-1}=o\big ((nh_n)^{1/2}h_n^{(3+\alpha^{-1})/J}\big )$. Similar inclusions are valid for the sets $\mathcal{E}^-_0$ and $\hat{\mathcal{E}}^-_0$. \\

     As an immediate consequence we have that $s_n(t,s)=\hat s_n(t,s)$ with high probability on the sets $\hat{\mathcal{E}}^+_{\rho_n},\hat{\mathcal{E}}^-_{\rho_n}, \mathcal{E}^+_{\rho_n},\mathcal{E}^-_{\rho_n}, \mathcal{E}_{\rho_n, P}$ whenever it holds that $d_\infty>0$.
\end{Lemma}
\begin{proof}
Using Lemma \ref{pl3} and (A6) we only need to show that 
\begin{align}
\label{p70}
    \left\|\mathcal{H}_n(t,s)\right\|_{\infty, I_n\times [0,1]}=O_\p\big (h_n^{-(3+\alpha^{-1})/J}\big )
\end{align}

Let $Q_n$ be any partition of $I_n \times [0,1]$ with meshes whose side lengths are proportional to $h_n^{1/\alpha}$ and $h_n^2$, respectively.  We then have 
\begin{align*}    
    \norm{\mathcal{H}_n(t,s)}_\infty
   \lesssim &  \sup_{(t,s) \in Q}\Big|\mathcal{H}_n(t,s)\Big|+\sup_{{\substack{|s-s'|^\alpha\leq h_n\\ \   |t-t'|h_n^{-1}\leq h_n}}}\Big|\mathcal{H}_n(t,s)-\mathcal{H}_n(t',s')\Big| \\
    \lesssim& \sup_{(t,s) \in Q}\Big|\mathcal{H}_n(t,s)\Big|+O_\p\big (h_n^{-(3+\alpha^{-1})/J}\big) 
\end{align*}
where we used Lemma \ref{pl1} for the second inequality. Using Theorem 3 in \cite{Yoshi78} combined with Lemma 2.2.2 from \cite{Wellner1996} additionally gives
\begin{align}
    \sup_{(t,s) \in Q}\Big|\mathcal{H}_n(t,s)\Big|= O_\p(|Q_n|^{1/J})=O_\p\Big(h_n^{-(2+1/\alpha)/J}\Big)
\end{align}
which  yields \eqref{p70} and finishes the proof.\\

For the additional implication in the Lemma statement we exemplarily note that 
\begin{align}
    \p(\hat{\mathcal{E}}^+_0 \subset \mathcal{E}^+_{\rho_n}) &= \p\Big( \sup_{(t,s)\in I_n\times [0,1]}(\tilde \mu_{h_n}(t,s)-\hat g_n(s)) \geq   \sup_{(t,s)\in I_n\times [0,1]}(\mu(t,s)-g_\mu(s))-\rho_n\Big)\\
    &\geq \p\Big(\|\tilde \mu_{h_n}(t,s)-\hat g_n(t,s)-(\mu(t,s)-g_\mu(t,s))\|_{\infty,I_n \times [0,1]}\leq \rho_n\Big)~,
\end{align}
which yields the desired result by the bounds for 
$$\|\tilde \mu_{h_n}(t,s)-\hat g_n(t,s)-(\mu(t,s)-g_\mu(t,s))\|_{\infty,I_n \times [0,1]}
$$  that we have  established.

\end{proof}

\subsection{Block Bootstrap Approximation}
This section contains an auxiliary result that will be essential for the proof of Theorem \ref{pt4} but is also of own interest. In addition to  Assumptions (A1)-(A6) we also assume (A7).

\begin{theorem}
\label{pt3}
     Under the assumptions of Theorem \ref{pt4} we have for any sequence $\rho_n$ and any $x_0>0$ that  
    \begin{align}
   &    \sup_{ y\geq x_0}\Big | \p\Big (\sup_{(t,s) \in \mathcal{E}_{\rho_n}}s(t,s)\mathcal{H}_n(t,s) \geq y \Big ) \\
       & ~~~~~~~~~~~~~
       -  \p\Big (\sup_{(t,s) \in \mathcal{E}_{\rho_n,P}}\Big(\frac{s(t,s)}{\sqrt{mq}}\sum_{l=1}^m\nu_l\sum_{j \in I_l} \hat \epsilon_j(s)h_n^{-1/2}K^*_{h_n}(j/n-t)\Big)>y\Big|\mathcal{Y}\Big )\Big |=o(1)~.
\end{align}
holds on a set with probability $1-o(1)$, where $\mathcal{Y}=\sigma(X_1,...,X_n)$.
\end{theorem}
\begin{proof}
 
Using Lemma \ref{pl6} we obtain for any $x_0>0$ that
\begin{align}
 &  \sup_{ y\geq x_0}\Big | \p\Big  (\sup_{(t,s) \in \mathcal{E}_{\rho_n,P}}(s(t,s)\mathcal{H}_n(t,s)) \geq y \Big )
    \\
       & ~~~~~~~~~~~~~
       -\p\Big( \sup_{(t,s) \in \mathcal{E}_{\rho_n,P}}\Big(\frac{s(t,s)}{\sqrt{mq}}\sum_{l=1}^m\nu_l\sum_{j \in I_l}\epsilon_j(s)h_n^{-1/2}K^*_{h_n}(j/n-t)\Big) \geq y| \mathcal{Y}\Big)\Big |\\
        & ~~~~~~~~~~~~~
    \leq n^{-c}+(m-1)\beta(r) \label{p31}
\end{align}
holds on a set $\mathcal{B}$ that has probability at least $1-o(1)-(m-1)\beta(r)$. Further we have by Lemma \ref{pl8} that uniformly in $y \in \R$
\begin{align}
\label{p32}
     &\p\Big (\sup_{(t,s) \in \mathcal{E}_{\rho_n,P}}\Big(\frac{s(t,s)}{\sqrt{mq}}\sum_{l=1}^m\nu_l\sum_{j \in I_l}\hat \epsilon_j(s)h_n^{-1/2}K^*_{h_n}(j/n-t)\Big)>y\Big|\mathcal{Y}\Big )=\\     
     & ~~~~~ ~~ \p\Big (\sup_{(t,s) \in \mathcal{E}_{\rho_n,P}}\Big(\frac{s(t,s)}{\sqrt{mq}}\sum_{l=1}^m\nu_l\sum_{j \in I_l} \epsilon_j(s)h_n^{-1/2}K^*_{h_n}(j/n-t)\Big)>y\Big|\mathcal{Y}\Big )+o(1)
\end{align}
holds on a set $\mathcal{A}$ with $\p(\mathcal{A})=1-o(1)$.  Combining \eqref{p31} with  \eqref{p32} therefore yields that on the set $\mathcal{A}\cap \mathcal{B}$  
\begin{align}
   &  \sup_{ y\geq x_0}\Big  | \p\Big (\sup_{(t,s) \in \mathcal{E}_{\rho_n,P}}(s(t,s)\mathcal{H}_n(t,s)) \geq y \Big )
     \\
        & ~~~~~~~~~~~~~-  \p\Big (\sup_{(t,s) \in \mathcal{E}_{\rho_n,P}}\Big(\frac{s(t,s)}{\sqrt{mq}}\sum_{l=1}^m\nu_l\sum_{j \in I_l} \hat \epsilon_j(s)h_n^{-1/2}K^*_{h_n}(j/n-t)\Big)>y\Big|\mathcal{Y}\Big )\Big  |=o(1)~,
\end{align}
which combined with arguments similar to those following equation \eqref{p25}(use Lemma \ref{pl5} to establish the analogue to \eqref{p25} in this case) yields that on the set $\mathcal{A}\cap \mathcal{B}$ 
\begin{align}
     &  \sup_{ y\geq x_0}\Big | \p\Big (\sup_{(t,s) \in \mathcal{E}_{\rho_n}}(s(t,s)\mathcal{H}_n(t,s)) \geq y \Big )\\
        & ~~~~~~~~~~~~~-  \p\Big (\sup_{(t,s) \in \mathcal{E}_{\rho_n,P}}\Big(\frac{s(t,s)}{\sqrt{mq}}\sum_{l=1}^m\nu_l\sum_{j \in I_l}\hat  \epsilon_j(s)h_n^{-1/2}K^*_{h_n}(j/n-t)\Big)>y\Big|\mathcal{Y}\Big  )\Big |=o(1)~.
\end{align}

\end{proof}

\begin{Lemma}
\label{pl5}
    We have for any sequence $\rho_n $ that   
    \begin{align}
    \Big |\sup_{(t,s)\in \mathcal{E}_{\rho_n}}(s(t,s)\mathcal{H}_n(t,s))-\sup_{(t,s) \in \mathcal{E}_{\rho_n,P}}(s(t,s)\mathcal{H}_n(t,s))\Big  | \leq O_\p(n^{-1/2})~.
\end{align}

\end{Lemma}
\begin{proof}
    Using the same arguments as in the proof of Lemma \ref{pl1} yields that with high probability
\begin{align}
    \Big |\sup_{(t,s)\in \mathcal{E}_{\rho_n}}(s(t,s)\mathcal{H}_n(t,s))-\sup_{(t,s) \in \mathcal{E}_{\rho_n,P}}(s(t,s)\mathcal{H}_n(t,s))\Big | &\lesssim \sup_{{\substack{|s-s'|\leq n^{-1/\alpha}\\ \   |t-t'|\leq n^{-1}}}}\Big|\mathcal{H}_n(t,s)-\mathcal{H}_n(t',s')\Big|\\
    &\lesssim n^{-1/2}~.
\end{align}

\end{proof}

\begin{Lemma}
    \label{pl6}
    Under the assumptions of Theorem \ref{pt4} we have for any sequence $\rho_n$ and any  $x_0>0$ that  
    \begin{align}
    \label{p22}
&    \sup_{ y\geq x_0}\Big | \p\Big (\sup_{(t,s) \in \mathcal{E}_{\rho_n,P}}(s(t,s)\mathcal{H}_n(t,s)) \geq y\Big ) \\
    & ~~~~~~~~~~~
    -\p\Big ( \sup_{(t,s) \in \mathcal{E}_{\rho_n,P}}\Big(\frac{s(t,s)}{\sqrt{mq}}\sum_{l=1}^m\nu_l\sum_{j \in I_l}\epsilon_j(s)h_n^{-1/2}K^*_{h_n}(j/n-t)\Big) \geq y \Big| \mathcal{Y}\Big) \Big| \\
    & ~~~~~~~~~~~
    \leq n^{-c}+2(m-1)\beta(r)
\end{align}
for some $c>0$, with probability at least $1-o(1)-(m-1)\beta(r)$.
\end{Lemma}
\begin{proof}
    By Lemma \ref{pl11}, which is stated and proved in Section \ref{sec6},  we may  assume that the constant $D_n$ in 
    Theorem \eqref{pt10} satisfies  $D_n \lesssim \log(n)h_n^{-1/2}$. 
    We obtain the desired statement by combining Theorem \eqref{pt10} with Theorem E.2 from \cite{Chernozhukov2018}. Note that  Theorem E.2 in this reference is not directly applicable, but we may use Lemma \ref{pl12} instead of Theorem 2 from \cite{Chernuzhokov2013b} in its proof to obtain the required result for our setting.

\end{proof}

\begin{Lemma}
    \label{pl8}
    We have for any sequence $\rho_n$ that on a set $\mathcal{A}$ with $\p(\mathcal{A})=1-o(1)$ that
    \begin{align}
     &\p\Big(\sup_{(t,s) \in \mathcal{E}_{\rho_n,P}}\Big(\frac{s(t,s)}{\sqrt{mq}}\sum_{l=1}^m\nu_l\sum_{j \in I_l}\hat \epsilon_j(s)h_n^{-1/2}K^*_{h_n}(j/n-t)\Big)>y\Big|\mathcal{Y}\Big)=\\     
     & \p\Big(\sup_{(t,s) \in \mathcal{E}_{\rho_n,P}}\Big(\frac{s(t,s)}{\sqrt{mq}}\sum_{l=1}^m\nu_l\sum_{j \in I_l} \epsilon_j(s)h_n^{-1/2}K^*_{h_n}(j/n-t)\Big)>y\Big|\mathcal{Y}\Big)+o(1)
\end{align}
where the $o(1)$ Term does not depend on $y$.
\end{Lemma}
\begin{proof}
    We observe  by Lemma \ref{pl2} in combination with (A6) that 
\begin{align}
    \max_{h_n^{-1} \leq j \leq n-h_n^{-1}}\norm{\epsilon_j(\cdot)-\hat \epsilon_j(\cdot)}_\infty& \leq \norm{\tilde \mu_{h_n}(\cdot,\cdot)-\hat g_n(\cdot)-\Big(\mu(\cdot,\cdot)-g_\mu(\cdot)\Big)}_{\infty, I_n \times [0,1]} \\
    & \quad +o_\p\Big((nh_n)^{-1/2}h_n^{-(3+1/\alpha)/J}\Big)\\    
    &\lesssim  \sqrt{\log(n)}\Big((nh_n)^{-1/2}h_n^{-(3+1/\alpha)/J}\Big)
\end{align}
holds on a set $\mathcal{A}$ with $\p(\mathcal{A})=1-o(1)$.  \\

In particular we have, letting $w_n=\Big((nh_n)^{-1/2}h_n^{-(3+\alpha^{-1})/J}\Big)$, that on $\mathcal{A}$ the random variables 
\begin{align}
\label{p501}
\nu_lh_n^{-1/2}/\sqrt{q}\sum_{i \in I_l}(\epsilon_i(s)-\hat \epsilon_i(s))K^*_{h_n}(i/n-t) \quad \quad 1+\ceil{qh_n} \leq l \leq n-\ceil{qh_n}
\end{align} are Subgaussian conditional on $\mathcal{Y}$, with variance proxy at most $C\frac{q\log(n)w_n^2}{h_n}$. This implies that on a $\mathcal{A}$ we have that
\begin{align}
 &\p\Big(\sup_{(t,s) \in \mathcal{E}_{\rho_n,P}}\Big(\frac{s(t,s)}{\sqrt{mq}}\sum_{l=1+\ceil{qh_n} }^{m-\ceil{qh_n} }\nu_l\sum_{j \in I_l}(\epsilon_j(s)-\hat \epsilon_j(s))h_n^{-1/2}K^*_{h_n}(j/n-t)\Big)>y\Big|\mathcal{Y}\Big) \\
 &\leq n^{1+1/\alpha}\exp(-y^2h_n/(2Cq\log(n)w_n^2))~.
\end{align}
Choosing $y=\sqrt{(4+2/\alpha)Cw_n^2q\log^2(n)h_n^{-1}}\leq n^{-c}$ yields that on the set $\mathcal{A}$ we have 
\begin{align}
\label{p25}
    &\p\Big(\sup_{(t,s) \in \mathcal{E}_{\rho_n,P}}\Big(\frac{s(t,s)}{\sqrt{mq}}\sum_{l=1+\ceil{qh_n}}^{m-\ceil{qh_n}}\nu_l\sum_{j \in I_l}(\epsilon_j(s)-\hat \epsilon_j(s))h_n^{-1/2}K^*_{h_n}(j/n-t)\Big)> n^{-c/2}\Big) \\
    &  \leq 1/n~.
\end{align}
Using (A7) and the arguments in the proof of Theorem 4.2 of \cite{Chetverikov2020} allows us to apply Theorem 1 from \cite{chernozhukov2017} with $1/\log(n) \lesssim \underline \sigma$ to the random variables \eqref{p501}. In combination with \eqref{p25} this yields that on the set $\mathcal{A}$ we have
\begin{align}
\label{p500}
     &\p\Big(\sup_{(t,s) \in \mathcal{E}_{\rho_n,P}}\Big(\frac{s(t,s)}{\sqrt{mq}}\sum_{l=1+\ceil{qh_n} }^{m-\ceil{qh_n} }\nu_l\sum_{j \in I_l}\hat \epsilon_j(s)h_n^{-1/2}K^*_{h_n}(j/n-t)\Big)>y\Big|\mathcal{Y}\Big)\\
     \geq & \p\Big(\sup_{(t,s) \in \mathcal{E}_{\rho_n,P}}\Big(\frac{s(t,s)}{\sqrt{mq}}\sum_{l=1+\ceil{qh_n} }^{m-\ceil{qh_n} }\nu_l\sum_{j \in I_l}\epsilon_j(s)h_n^{-1/2}K^*_{h_n}(j/n-t)\Big)>y+n^{-c/2}\Big|\mathcal{Y}\Big)+o(1)\\
     =& \p\Big(\sup_{(t,s) \in \mathcal{E}_{\rho_n,P}}\Big(\frac{s(t,s)}{\sqrt{mq}}\sum_{l=1+\ceil{qh_n} }^{m-\ceil{qh_n} }\nu_l\sum_{j \in I_l} \epsilon_j(s)h_n^{-1/2}K^*_{h_n}(j/n-t)\Big)>y\Big|\mathcal{Y}\Big)+o(1)
\end{align}
where the $o(1)$ Term does not depend on $y$. An analogous argument establishes the reverse inequality.\\
Similar arguments can be used for the first and last $\ceil{qh_n}$ random variables that we did not consider in \eqref{p501} and the arguments following it. In this case Lemma \ref{pl2} is not available, but a similar result can be established by using Lemma B.1 from \cite{dette2015} instead. This establishes the desired result.

\end{proof}

\subsection{Proof of Theorem \ref{pt4}}

\textbf{Proof of \eqref{p44}}
\begin{proof}
We focus on the case $d_\infty>0$, the arguments for $d_\infty=0$ are similar but easier.    

By \ref{pt2} we have that
\begin{align}
\label{p40}
    \sqrt{nh_n}(\hat d_{\infty,n}-d_{\infty,n}) \leq \sup_{(t,s) \in \mathcal{E}_{\rho_n/2}} \Big(s(t,s)\mathcal{H}_n(t,s)\Big)+o_\p(n^{-\gamma_2})~.
\end{align}
Further we know by the definitions of $\mathcal{E}_{\rho,P}$ and $\hat{\mathcal{E}}_{\rho,P}$ and by Lemma \ref{pl2} that with high probability it holds that $\mathcal{E}_{\rho_n/2,P}\subset \hat{\mathcal{E}}_{\rho_n,P}$ which implies that
\begin{align}
\label{p42}
    &\sup_{(t,s) \in \mathcal{E}_{\rho_n/2,P}}\Big( \frac{s(t,s)}{\sqrt{mq}}\sum_{l=1}^m\nu_l\sum_{j \in I_l}\epsilon_j(s)h_n^{-1/2}K^*_{h_n}(j/n-t) \Big)\\
    \leq&\sup_{(t,s) \in \hat{\mathcal{E}}_{\rho_n,P}}\Big(\frac{s(t,s)}{\sqrt{mq}}\sum_{l=1}^m\nu_l\sum_{j \in I_l}\epsilon_j(s)h_n^{-1/2}K^*_{h_n}(j/n-t) \Big)
\end{align}
holds with high probability. Denote by $\hat q^*_{1-\alpha}$ and $q^*_{1-\alpha}$ the $(1-\alpha)$-quantiles of 
\begin{align}
    \sup_{(t,s) \in \hat{\mathcal{E}}_{\rho_n,P}}\Big(\frac{s(t,s)}{\sqrt{mq}}\sum_{l=1}^m\nu_l\sum_{j \in I_l}\epsilon_j(s)h_n^{-1/2}K^*_{h_n}(j/n-t) \Big)
\end{align} and
\begin{align}
    \sup_{(t,s) \in \mathcal{E}_{\rho_n/2,P}}\Big( \frac{s(t,s)}{\sqrt{mq}}\sum_{l=1}^m\nu_l\sum_{j \in I_l}\epsilon_j(s)h_n^{-1/2}K^*_{h_n}(j/n-t) \Big)~,
\end{align} 
respectively. Note that by Lemma \ref{pl2} the definition of $\hat q_{1-\alpha}^*$ given here agrees with the definition given in \eqref{p100} with high probability. By \eqref{p42} we clearly have $q^*_{1-\alpha}\leq \hat q^*_{1-\alpha}$ with high probability.\\

By Theorem \ref{pt3} we have for any $x_0>0$ that 
\begin{align}
\label{p41}
       & \sup_{ y\geq x_0}\Big| \p\Big(\sup_{(t,s) \in \mathcal{E}_{\rho_n/2}}(s(t,s)\mathcal{H}_n(t,s))\geq y \Big) \\
       & ~~~~~~~~~~~~~
       - \p\Big(\sup_{(t,s) \in \mathcal{E}_{\rho_n/2,P}}\Big(\frac{s(t,s)}{\sqrt{mq}}\sum_{l=1}^m\nu_l\sum_{j \in I_l} \hat \epsilon_j(s)h_n^{-1/2}K^*_{h_n}(j/n-t)\Big)>y\Big|\mathcal{Y}\Big)\Big|      =o(1)~.
\end{align}

Now combining the above observations we obtain with high probability that 
\begin{align}
    &\p( T_{n,\Delta} > \hat q^*_{1-\alpha})\\
    \leq &\p(  \sqrt{nh_n}(\hat d_{\infty,n}-d_{\infty,n}) > \hat q^*_{1-\alpha}) \\
    \leq& \p(  \sqrt{nh_n}(\hat d_{\infty,n}-d_{\infty,n}) >  q^*_{1-\alpha})+o(1)\\
    \leq& \p\Big(\sqrt{nh_n}\sup_{(t,s) \in \mathcal{E}_{\rho_n/2,P}}\Big(s(t,s)\mathcal{H}_n(t,s)\Big)+n^{-\gamma_2}>q^*_{1-\alpha}\Big)+o(1)\\
    =& \p\Big(\sup_{(t,s) \in \mathcal{E}_{\rho_n/2,P}}\Big(\frac{s(t,s)}{\sqrt{mq}}\sum_{l=1}^m\nu_l\sum_{j \in I_l} \hat \epsilon_j(s)h_n^{-1/2}K^*_{h_n}(j/n-t)\Big)>q^*_{1-\alpha}-n^{-\gamma_2}\Big|\mathcal{Y}\Big)\\
    &~~~~~~~~~+o(1)\\
    =&  \p\Big(\sup_{(t,s) \in \mathcal{E}_{\rho_n/2,P}}\Big(\frac{s(t,s)}{\sqrt{mq}}\sum_{l=1}^m\nu_l\sum_{j \in I_l}  \epsilon_j(s)h_n^{-1/2}K^*_{h_n}(j/n-t)\Big)>q^*_{1-\alpha}\Big|\mathcal{Y}\Big)+o(1)\\
    =&\alpha+o(1)~,
\end{align}
where the third inequality follows by \eqref{p40}, the forth follows by \eqref{p41}
and the fifth by  equation \eqref{p25}  in combination with Theorem 1 from \cite{chernozhukov2017} (Here we may assume $1/\log(n) \lesssim \underline \sigma$ by (A7) and the arguments in the proof of Theorem 4.2 of \cite{Chetverikov2020}). 
\end{proof}

\textbf{Proof of \eqref{p45}}
Note that
\begin{align}
\label{p48}
    \hat T_{n,\Delta}=\sqrt{nh_n}(\hat d_{\infty,n}-d_{\infty,n})+\sqrt{nh_n}(d_{\infty,n}-\Delta)~.
\end{align}

By Lemmas \ref{pl4} and \ref{pl2} and assumption (A6) we have 
\begin{align}
\label{p46}
    \sqrt{nh_n}(\hat d_{\infty,n}-d_{\infty,n}) \lesssim O_\p\Big(h_n^{-(3+\alpha^{-1})/J}\Big)~.
\end{align}
Further we observe by the continuity of $\mu$ and $g_\mu$ that for $n$ sufficiently large we have
\begin{align}
    d_\infty-d_{\infty,n}<(d_\infty-\Delta)/2
\end{align}
so that for $n$ large enough
\begin{align}
\label{p47}
    \sqrt{nh_n}(d_{\infty,n}-\Delta) \geq \sqrt{nh_n}(d_\infty-\Delta)/2~.
\end{align}
Combining \eqref{p48}, \eqref{p46} and \eqref{p47} gives
\begin{align}
     T_{n,\Delta} \geq  O_\p\Big(h_n^{-(3+\alpha^{-1})/J}\Big)+ \sqrt{nh_n}(d_\infty-\Delta)/2
\end{align}
which yields the desired result.

\subsection{Proof of Theorem \ref{pt5}}

We start with the case $t^*(s,\Delta) \in [x_0,x_1]$. By Lemma \ref{pl2} and the choice of $\delta_n$ we always have, on a set $\mathcal{A}$ with $\p(\mathcal{A})=1-o(1)$ that does not depend on $s$, that
\begin{align}
    \sup_{t \in [x_0\lor h_n, y]}|\hat d(t,s)|<\Delta-\delta_n
\end{align} holding implies that
\begin{align}
    \sup_{t \in [x_0, y]}| d(t,s)|<\Delta
\end{align}
holds. By the representations \eqref{p60} and \eqref{p61} we then immediately obtain for all $s \in S$ that
\begin{align}
\label{p62}
    \hat t^*(s,\Delta) \leq t^*(\Delta)+h_n
\end{align}
holds on $\mathcal{A}$.\\

Similarly we have that
\begin{align}
     d(t^*(\Delta,s),s)=\Delta
\end{align}
implies that for any $x>0$ we have
\begin{align}
    d(t^*(\Delta,s)-x,s) \geq \Delta - \omega(x,s)~.
\end{align}
Choosing $x=\omega^{-1}(\delta_n/2,s)$ and applying Lemma \ref{pl2} then implies for all $s \in S$ that
\begin{align}
    \hat d(t^*(\Delta,s)-x,s)\geq \Delta -\delta_n/2-O_\p\Big((nh_n)^{-1/2}h_n^{-(3+\alpha^{-1})/J}\Big)\geq \Delta -\delta_n
\end{align}
holds on $\mathcal{A}$. A similar argument for the case $ d(t^*(\Delta,s),s)=-\Delta$ and again using the representations \eqref{p60} and \eqref{p61} yields that for all $s \in S$ we have that
\begin{align}
\label{p63}
    \hat t^*(\Delta,s) \geq t^*(\Delta,s)-\omega^{-1}(\delta_n/2,s)
\end{align}
holds on $\mathcal{A}$. Combining \eqref{p62} and \eqref{p63} yields the first part of Theorem \ref{pt5}.\\

Let us now turn to the case $t^*(s,\Delta)=\infty$. Define $D_{\infty,n}(s)=\sup_{ t \in I_n}|d(t,s)|$ and $\hat D_{\infty,n}(s)=\sup_{ t \in I_n}|\hat d(t,s)|$ and observe that
\begin{align}
    \p(\hat t^*(s,\Delta) < t^*(s,\Delta))&= \p( \hat D_{\infty,n}(s)>\Delta-\delta_n)\\
    &=\p(\hat D_{\infty,n}(s) - D_{\infty,n}(s)+\delta_n > \Delta - D_{\infty,n}(s))\\
    &=o(1)
\end{align}
due to Lemma \ref{pl2}, the choice of $\delta_n$ and the fact that $\Delta - D_{\infty,n}(s)>0 
 \ \forall s \in S^c$. Note that we can make this uniform in $s$ in the case that $\Delta-d_\infty>0$ which finishes the proof of the second part of Theorem \ref{pt5}.

\subsection{Gaussian approximation for dependent data}

In this section we will consider a strictly stationary mean zero sample $X_1,...,X_n \in \R^p$ that we assume to be $\beta$-mixing. Further let $T_n=\max_{1 \leq j \leq p} n^{-1/2}\sum_{i=1}^nX_{ij}$. Assume that there exists some constant $D_n$ such that 
\begin{align}
    |X_{ij}|\leq D_n \quad \text{ a.s. } 1 \leq i \leq n, 1 \leq j \leq p
\end{align}
Further let 
\begin{align}
    B_l=\sum_{ i \in I_l} X_i, \quad S_l= \sum_{ i \in J_l} X_i
\end{align}
and $\{ \tilde B_l , 1 \leq l \leq m\}$, $\{ \tilde S_l, 1\leq l \leq m\}$ be sequences of independent copies of $B_l$ and $S_l$. Denote by $B_{li}$ the $i$-th coordinate of $B_l=(B_{l1},...,B_{lp})$. Additionally let $Y=(Y_1,...,Y_p)$ be a centered normal vector with covariance matrix $\E[YY^T]=(mq)^{-1}\sum_{i=1}^m\E[B_iB_i^T]$. \\

For any integer $1 \leq q \leq n$ we also define
\begin{align}
    \bar \sigma(q):=\max_{1 \leq j \leq p }\frac{1}{m}\sum_{I \in \{I_1,...,I_m\}}\text{Var}\Big(q^{-1/2}\sum_{i \in I}X_{ij}\Big)   
\end{align}
Additionally let $\bar \sigma(0):=1$.
We now establish a Gaussian approximation result for $\beta$-mixing random variables that allows for decaying variances in all but one coordinate.

\begin{theorem}
\label{pt10}
   Suppose that 
   \begin{align}
       \bar \sigma(q) &\gtrsim 1 \\
       \bar \sigma(q) \lor \bar \sigma(r) &\lesssim 1\\   
       (q+r\log(pn))\log(pn)^{1/2}D_nn^{-1/2}&\lesssim n^{-c_1}\\
       r/q \log(p)^2 & \lesssim n^{-c_1}\\
       q^{1/2}D_n\log(pn)^5n^{-1/2}&\lesssim n^{-c2}       
   \end{align}
   then
   \begin{align}
     \sup_{t >x_0}\Big|\p(T_n\leq t)-\p\Big(\max_{1 \leq j \leq p}Y_j\leq t\Big)\Big|\lesssim n^{-c}+(m-1)\beta(r)
\end{align}
for some $c>0$ only depending on $c_1,c_2$.  In particular we also have that
\begin{align}
\label{p300}
     \sup_{t >x_0}\Big|\p(T_n\leq t+\epsilon)-\p(T_n\leq t)\Big|\lesssim \epsilon\sqrt{\log(pn)}+n^{-c}+(m-1)\beta(r)
\end{align}
for some $c>0$.
\end{theorem}
\begin{proof}
For the remainder of this proof $c$ denotes a generic positive constant that may change from line to line that only depends on $c_1,c_2$ as well as $d_1,d_2$.  First we bound the error incurred by leaving out the small blocks, i.e. we have that
    \begin{align}
        |T_n-\max_{1 \leq j \leq p}n^{-1/2}\sum_{i=1}^mB_{ij}|\leq \max_{1 \leq j \leq p}\Big|n^{-1/2}\sum_{i=1}^mS_{ij}\Big| + \max_{1 \leq j \leq p}\Big|n^{-1/2}S_{(m+1)j}\Big|
    \end{align}
and note that
\begin{align}  
\label{p149}
    \max_{1 \leq j \leq p}\Big|n^{-1/2}S_{(m+1)j}\Big|&\lesssim  qD_n/\sqrt{n}     
\end{align}

By Corollary 2.7 from \cite{Yu1994} (or Lemma 1 from \cite{Yoshihara1978}) we have
\begin{align}
\label{p150}
   &\Big| \p\Big(\max_{1 \leq j \leq p}n^{-1/2}\sum_{i=1}^mB_{ij}\leq t\Big)-\p\Big(\max_{1 \leq j \leq p}n^{-1/2}\sum_{i=1}^m\tilde B_{ij}\leq t\Big)\Big| \leq (m-1)\beta(r)\\
   &\Big| \p\Big(\max_{1 \leq j \leq p}n^{-1/2}\Big|\sum_{i=1}^mS_{ij}\Big|\leq t\Big)-\p\Big(\max_{1 \leq j \leq p}n^{-1/2}\Big|\sum_{i=1}^m\tilde S_{ij}\Big|\leq t\Big)\Big| \leq (m-1)\beta(q)
\end{align}
Using independence and the fact that $\tilde S_{ij}$ is bounded by $rD_n$ we obtain by Lemma D.3 from \cite{Chetverikov2020} that
\begin{align}
    \E\left[\max_{1 \leq j \leq p}\Big|n^{-1/2}\sum_{i=1}^m\tilde S_{ij}\Big|\right]\lesssim \sqrt{r/q \bar \sigma(r)\log(p)}+n^{-1/2}rD_n\log(p) 
\end{align}
so that 
\begin{align}
\label{p151}
     \p\Big(\max_{1 \leq j \leq p}n^{-1/2}\Big|\sum_{i=1}^m\tilde S_{ij}\Big|> t\Big) &\lesssim \Big(\sqrt{r/q \bar \sigma(r)\log(p)}+n^{-1/2}rD_n\log(p)\Big)/t   
\end{align}

Let $\delta_1=\sqrt{n^{-c_1}/\log(pn)}$. Combining \eqref{p149}, \eqref{p150}, \eqref{p151} (choose $t=\delta_1$ in \eqref{p151}) we obtain for $t>x_0$ that
\begin{align}
    \p(T_n\leq t)& \leq  \p\Big(\max_{1 \leq j \leq p}n^{-1/2}\sum_{i=1}^mB_{ij}\leq t+\delta_1+qD_nn^{-1/2}\Big)+n^{-c}+(m-1)\beta(q)\\
    & \leq \p\Big(\max_{1 \leq j \leq p}n^{-1/2}\sum_{i=1}^m\tilde B_{ij}\leq t+\delta_1+qD_nn^{-1/2}\Big)+n^{-c}+(m-1)\beta(q)\\
    & \leq  \p\Big(\max_{1 \leq j \leq p}n^{-1/2}\sum_{i=1}^m\tilde B_{ij}\leq t\Big)+n^{-c}+ \Big(\frac{qD_n^2\log(pn)^{10}}{n}\Big)^{1/6}+(m-1)\beta(r)
\end{align}
where we used Lemma \ref{pl10} for the last line. A similar argument also yields the reverse inequality so that for all $t>x_0$ we have
\begin{align}
   \sup_{t >x_0}\Big|\p(T_n\leq t)-\p\Big(\max_{1 \leq j \leq p}n^{-1/2}\sum_{i=1}^m\tilde B_{ij}\leq t\Big)\Big|\lesssim n^{-c}+(m-1)\beta(r)
\end{align}

Now we apply Theorem 4.1 from \cite{Chetverikov2020}  (note that for at least one $j$ we have $\bar \sigma(q) \simeq n^{-1}\sum_{i=1}^m\text{Var}(\tilde B_{ij})$ and that for all $j$ it holds that $\tilde B_{ij}\lesssim qD_n$) to further obtain that
\begin{align}
\label{p301}
     \sup_{t >x_0}\Big|\p(T_n\leq t)-\p\Big(\max_{1 \leq j \leq p}\sqrt{mq/n}Y_j\leq t\Big)\Big|\lesssim n^{-c}+(m-1)\beta(r)
\end{align}
The result then follows by the same arguments as in step 4 of the proof of Theorem E.1 from \cite{Chernozhukov2018} where we use Lemma C.3 from \cite{Chetverikov2020} instead of Step 3.

The additional conclusion of the theorem follows by combining equation \eqref{p301} with Lemma \ref{pl10}

\end{proof}

\begin{Lemma}
\label{pl10}
     Suppose that the conditions of Theorem \ref{pt10} are fulfilled for $q=1,r=0$ and that $X_1,...,X_n$ are independent. Then for for any $x_0>0$ we have that for all $t>x_0$ and $\epsilon>0$ it holds that
    \begin{align}
        &\p\Big(\sup_{1 \leq j \leq p}n^{-1/2}\sum_{i=1}^nX_{ij}\leq t+\epsilon\Big)-\p\Big(\sup_{1 \leq j \leq p}n^{-1/2}\sum_{i=1}^nX_{ij}\leq t\Big)\\
        \lesssim &\Big(\frac{D_n^2\log(pn)^{10}}{n}\Big)^{1/6}+\epsilon\sqrt{\log(pn)}
    \end{align}
\end{Lemma}
\begin{proof}
   The proof is the same as for Lemma 4.4 in \cite{Chernozhukov22} but we use Theorem 4.1, equation (56) and Lemma C.3 from \cite{Chetverikov2020} instead of Lemma 4.3 and J.3.
\end{proof}

\section{Further technical details} \label{sec6}

The following Lemma is used in the proof of Lemma \ref{pl6} to reduce to the case of bounded random variables. 
\begin{Lemma}
\label{pl11}
    Let $Z_1,...,Z_n \in \R^p$ be a sequence of random vectors with mixing coefficients $(\beta_Z(k))_{k \in \N}$ such that
    \begin{align}
    \label{p157}
        \sup_{1 \leq j \leq p}\sup_{1 \leq i \leq n}\E[\exp(|Z_{ij}|)]&\lesssim 1\\
        \sum_{k=1}^\infty (k+1)^{J/2-1}\beta(k)^{\frac{\delta}{J+\delta}}&<\infty~, 
    \end{align}
    for some even $J\geq 8$ and some $\delta>0$.  We define for any $C>0$ the vectors
    \begin{align}
         \tilde Z_{ij}=Z_{ij}\1\{|Z_{ij}|\leq n^{1/3}\}       ~.
    \end{align}
    Assuming that the vectors $\tilde Z_{i}-\E[\tilde Z_{i}]$ fulfill the conditions of Theorem \ref{pt10} for some choice of $r,q$ and that $\log(np)\simeq \log(n)$ we have for any $x_0>0$ that 
     \begin{align}
         \sup_{t \geq x_0}\Big|\p\Big( \sup_{1 \leq j \leq p} n^{-1/2}\sum_{i=1}^nZ_{ij} \leq t\Big) - \p\Big( \sup_{1 \leq j \leq p}n^{-1/2}\sum_{i=1}^n(\tilde Z_{ij}-E[\tilde Z_{ij})] \leq t\Big)\Big| \lesssim o(1)~.
    \end{align}   
    Further, recalling the notations introduced after Theorem \ref{pt2} we have with high probability that
    \begin{align}
    \label{p502}
       \sup_{t \geq x_0}& \Big|\p\Big( \sup_{1 \leq j \leq p}\frac{1}{\sqrt{mq}}\sum_{l=1}^m\nu_l\sum_{k \in I_l}Z_{kj} \leq t \Big| \mathcal{Z}\Big)- \\
       &~~~~~~~~\p\Big( \sup_{1 \leq j \leq p}\frac{1}{\sqrt{mq}}\sum_{l=1}^m\nu_l\sum_{k \in I_l}(\tilde Z_{kj}-\E[\tilde Z_{kj}) \leq t \Big| \mathcal{Z}\Big)\Big| \lesssim o(1)
    \end{align}
    where $\mathcal{Z}=\sigma\Big(Z_1,...,Z_n\Big)$.
\end{Lemma}
\begin{proof}
    We begin with the first statement. Define 
    \begin{align}        
        \check Z_{ij}=Z_{ij}\1\{|Z_{ij}|>n^{1/3}\}
    \end{align}
    and note that
    \begin{align}
        Z_{ij}- \tilde Z_{ij}=\check Z_{ij}
    \end{align}
    so that
    \begin{align}
    \label{p152}
        &\p\Big( \sup_{1 \leq j \leq p} n^{-1/2}\sum_{i=1}^nZ_{ij} \leq t\Big) \\
        \geq& \p\Big( \sup_{1 \leq j \leq p}n^{-1/2}\sum_{i=1}^n(\tilde Z_{ij}-E[\tilde Z_{ij})] \leq t-\sup_{1 \leq j \leq p} n^{-1/2}\sum_{i=1}^n(\check Z_{ij}-\E[\check Z_{ij}]) \Big)
    \end{align}    
    Consider for an $s$ to be chosen later the inequalities
    \begin{align}
     \label{p153}
         \p\Big( \sup_{1 \leq j \leq p} n^{-1/2}\sum_{i=1}^n(\check Z_{ij}-\E[\check Z_{ij}]) > s\Big) &\lesssim p\sup_{1 \leq j \leq p}\E\Big[ \Big(n^{-1/2}\sum_{i=1}^n(\check Z_{ij}-\E[\check Z_{ij}])\Big)^{J}\Big]/s^{J}\\
         &\lesssim p\sup_{1 \leq j \leq p}\sup_{1 \leq i \leq n}\E\Big[|\check Z_{ij}-\E[\check Z_{ij}]|^{J+\delta}\Big]^{J/(J+\delta)}/s^{J}\\
         &\lesssim p\sup_{1 \leq j \leq p}\sup_{1 \leq i \leq n}\norm{\check Z_{ij}}_{J+\delta}^{J}/s^{J} 
    \end{align}
    where the first inequality follows by Markov, the second by the arguments in the proof of Theorem 3 from \cite{Yoshi78} and the third due to Hölder's inequality. In particular we obtain, in view of \eqref{p157}, by applying Hölder's and Markov's inequalities that 
    \begin{align}
    \label{p156}
        \p\Big( \sup_{1 \leq j \leq p} n^{-1/2}\sum_{i=1}^n(\check Z_{ij}-\E[\check Z_{ij}]) > n^{-\epsilon}\Big)=o(1)
    \end{align}
    for any $\epsilon>0$.\\
    Next we  combine \eqref{p156} and \eqref{p300} applied to $\tilde Z_i-\E[\tilde Z_i], i=1,...,n$ to obtain for some $c>0$ that
    \begin{align}   
        &\Big|\p\Big( \sup_{1 \leq j \leq p}n^{-1/2}\sum_{i=1}^n(\tilde Z_{ij}-E[\tilde Z_{ij})] \leq t-\sup_{1 \leq j \leq p} n^{-1/2}\sum_{i=1}^n(\check Z_{ij}-\E[\check Z_{ij}]) \Big)\\
         & \quad \quad -\p\Big( \sup_{1 \leq j \leq p}n^{-1/2}\sum_{i=1}^n(\tilde Z_{ij}-E[\tilde Z_{ij})] \leq t\Big)\Big| \\
        &        \lesssim o(1)~.
    \end{align}
    Combining this with equation \eqref{p152} yields
    \begin{align}
         \p\Big( \sup_{1 \leq j \leq p} n^{-1/2}\sum_{i=1}^nZ_{ij} \leq t\Big) &\geq \p\Big( \sup_{1 \leq j \leq p}n^{-1/2}\sum_{i=1}^n(\tilde Z_{ij}-E[\tilde Z_{ij})] \leq t\Big)+o(1) 
    \end{align}
    A similar arguments yields the reverse inequality and finishes the proof of the first statement of the Lemma. For the second statement we merely note that the set $\{ Z_{ij}-\tilde Z_{ij}=0, i=1,...,n, j=1,...,p\}$ has probability $1-o(1)$, using this to bound $\bar \Delta$ when applying Lemma \ref{pl12} to the Gaussian vectors in \eqref{p502} yields the desired conclusion. 
\end{proof}

\begin{Lemma}
\label{pl12}
    Let $Z_1, Z_2$ be two $p$-dimensional mean zero Gaussian vectors with covariance matrices $\Sigma^1=(\sigma^1_{kl})_{1 \leq k,l \leq p}$ and $\Sigma^2=(\sigma^2_{kl})_{1 \leq k,l \leq p}$, respectively.  Then for any $x_0>0$ we have
    \begin{align}
        \sup_{t\geq x_0}\Big|\p\Big(\max_{1 \leq j \leq p}Z_{1j}\leq t\Big)-\p\Big(\max_{1 \leq j \leq p}Z_{1j}\leq t\Big)\Big|\lesssim \bar \Delta^{1/3}\log(np/\bar \Delta)
    \end{align}
    where 
    \begin{align}
        \bar \Delta=\sup_{1 \leq k,l \leq p}|\sigma^1_{kl}-\sigma^2_{kl}|~.
    \end{align}
\end{Lemma}
\begin{proof}
First use the arguments from the proof of Theorem 1 from \cite{Chetverikov2020} to reduce to the case $\E[Z^2_{1j}] \gtrsim \log(pn)^{-1}$. We now only need to make the implicit dependence of the constants on $\underline \sigma$ in the bound of Theorem 2 from  \cite{Chernuzhokov2013b} explicit to obtain the desired result. The dependence of the constants on $\underline \sigma$ is inherited from the use of Theorem 3 in its proof, we can make them explicit by using the bounds from Theorem 1 in \cite{chernozhukov2017} instead.
\end{proof}

\bigskip

\textbf{Acknowledgements.} 
This work was partially supported by the  Research unit 5381 {\it Mathematical Statistics in the Information Age}, project number 460867398 and by TRR 391
{\it Spatio-temporal Statistics for the Transition of Energy and Transport} (520388526)
funded by the Deutsche Forschungsgemeinschaft (DFG, German Research Foundation).
We would like to thank the reviewers for their constructive comments on an  earlier version of this paper.
\bibliographystyle{apalike}
\setlength{\bibsep}{2pt}
\bibliography{reference}

\end{document}